\documentclass[12pt]{amsart}
\usepackage[margin=1in,includeheadfoot]{geometry}
\usepackage{amsmath}
\usepackage{slashed}
\usepackage{graphicx}
\newtheorem{thm}{Theorem}[section]
\newtheorem{cor}[thm]{Corollary}
\newtheorem{lem}[thm]{Lemma}
\newtheorem{prop}[thm]{Proposition}
\theoremstyle{definition}

\theoremstyle{remark}
\newtheorem{rem}[thm]{Remark}
\numberwithin{equation}{section}

\newcommand{\jf}[2]{\int_{#1}{#2}}

\newcounter{stepnum}

\newcommand{\p}{\slashed\partial}

\begin{document}

\title[Energy identity for a class of approximate Dirac-harmonic maps]{Energy identity for a class of approximate Dirac-harmonic maps from surfaces with boundary}

\author[Jost]{J\"urgen Jost}
\address{Max Planck Institute for Mathematics in the Sciences\\ Inselstrasse 22\\ 04103 Leipzig, Germany}
\address{Department of Mathematics\\ Leipzig University\\ 04081 Leipzig, Germany}
\email{jost@mis.mpg.de}

\author[Liu]{Lei Liu}%
\address{Max Planck Institute for Mathematics in the Sciences\\ Inselstrasse 22 \\  04103 Leipzig,  Germany }
\email{leiliu@mis.mpg.de }%

\author[Zhu]{Miaomiao Zhu}
\address{School of Mathematical Sciences, Shanghai Jiao Tong University\\ 800 Dongchuan Road \\ Shanghai, 200240 \\China}
\email{mizhu@sjtu.edu.cn}

\thanks{M. Zhu was supported in part by National Natural Science Foundation of China (No. 11601325). We would like to thank
the referee for careful comments and useful suggestions in improving the presentation of the paper.}


\date{\today}
\begin{abstract}
For a sequence of coupled fields $\{(\phi_n,\psi_n)\}$ from a compact Riemann surface $M$ with smooth boundary to a general compact Riemannian manifold with uniformly bounded energy and satisfying the Dirac-harmonic system up to some uniformly controlled error terms, we show that the energy identity holds during a blow-up process near the boundary. As an application to the heat flow of Dirac-harmonic maps from surfaces with boundary, when such a flow blows up at infinite time, we obtain an energy identity.
\end{abstract}
\maketitle

\section{introduction}

\
This paper is a contribution to the study of coupled field equations on Riemann surfaces, merging the theory of harmonic maps from surfaces with a mathematical version of the nonlinear supersymmetric of quantum field theory. The corresponding action functional couples a term involving what is called the energy of a map from a surface to some Riemannian manifold with a Dirac action for a nonlinear spinor field. The solutions of the resulting Euler-Lagrange equations are called Dirac-harmonic maps \cite{chen2006dirac}. While they share many properties with harmonic maps, their analysis is much more subtle, because the Dirac action is not bounded from below. Therefore, standard variational methods do not apply to show the existence of solutions under general conditions. As an alternative, a new type of mixed parabolic-elliptic partial differential equations has been introduced \cite{CJSZ2014} and further investigated \cite{jost-Liu-Zhu-02} in order to develop new tools for the existence problem. The existence problem is still not fully solved. In order to make progress, results about the behavior at singularities that are known and classical for harmonic maps need to be extended to the Dirac-harmonic case. This is where the contribution of the present paper lies. We study the blow-up process and show a so-called energy identity, that is, all the energy that is removed from the map gets transferred to the bubbles that represent the singularity. In fact, we study this at the boundary, because boundary value problems currently offer the situation where the existence theory is best developed and most promising. In order that our results be applicable to the parabolic case, we have to consider approximate solutions, that is, fields that satisfy the Euler-Lagrange equations up to some controlled error term. This naturally makes the analysis more difficult.

We now fix the technical setting to describe our results in more precise terms.
Let $(M,h)$ be a compact connected Riemann surface with smooth boundary $\partial M$, equipped with a Riemannian metric $h$ and with a fixed spin structure, $\Sigma M$ be the spinor bundle over $M$ and $\langle\cdot,\cdot\rangle_{\Sigma M}$ be the metric on $\Sigma M$ induced by the Riemannian metric $h$. Choosing a local orthonormal basis ${e_\alpha,\alpha=1,2}$ on $TM$, the usual Dirac operator is
defined as $\slashed\partial:=e_\alpha\cdot\nabla_{e_\alpha}$, where $\nabla$ is the spin connection on $\Sigma M$, $\cdot$ is the Clifford multiplication, which satisfies the  skew-adjointness property  $$\langle X\cdot \psi_1,\psi_2\rangle_{\Sigma M}=-\langle \psi_1,X\cdot\psi_2\rangle_{\Sigma M}$$ for any $X\in\Gamma(TM)$, $\psi_i\in\Gamma(\Sigma M)$, $i=1,2$.

Let $\phi$ be a smooth map from $M$ to another compact Riemannian manifold $(N,g)$ with dimension $n\geq2$. Let $\phi^{*}TN$ be the pull-back bundle of $TN$ by $\phi$ and then we get the twisted bundle $\Sigma M\otimes\phi^{*}TN$. Naturally, there is a metric $\langle\cdot,\cdot\rangle_{\Sigma M\otimes\phi^{*}TN}$ on $\Sigma M\otimes\phi^{*}TN$ which is induced from the metrics on $\Sigma M$ and $\phi^{*}TN$. Also we have a natural connection $\widetilde{\nabla}$ on $\Sigma M\otimes\phi^{*}TN$ which is induced from the connections on $\Sigma M$ and $\phi^{*}TN$. Let $\psi$ be a section of the bundle $\Sigma M\otimes\phi^{*}TN$. In local coordinates, it can be written as
\[
\psi=\psi^i\otimes\partial_{y^i}(\phi),
\]
where each $\psi^i$ is a usual spinor on $M$ and ${\partial_{y^i}}$ is the nature local basis on $N$. Then $\widetilde{\nabla}$ becomes
\begin{eqnarray}\label{z2}
\widetilde{\nabla}\psi=\nabla\psi^i\otimes\partial_{y^i}(\phi)
+(\Gamma^i_{jk}\nabla\phi^j)\psi^k\otimes\partial_{y^i}(\phi),
\end{eqnarray}
where $\Gamma^i_{jk}$ are the Christoffel symbols of the Levi-Civita connection of $(N,g)$. The Dirac operator along the map $\phi$ is defined by $\slashed D\psi:=e_\alpha\cdot\widetilde\nabla_{e_\alpha}\psi$.

An important factor that will enable us to utilize tools from complex analysis is that the usual Dirac operator $\p$ on a surface can be seen as the Cauchy-Riemann operator. Consider $\mathbb{R}^2$ with the Euclidean metric $dx^2+dy^2$. Let $e_1=\frac{\partial}{\partial x}$ and $e_2=\frac{\partial}{\partial y}$ be the standard orthonormal frame. A spinor field is simply a map $\psi:\mathbb{R}^2\rightarrow\Delta_2=\mathbb{C}^2$, and the action of $e_1$ and $e_2$  on spinors can be identified with multiplication with matrices
\[
e_1=\begin{pmatrix}0 & 1 \\ -1 & 0\end{pmatrix}, \quad e_2=\begin{pmatrix}0 &i \\ i & 0\end{pmatrix}.
\]
If $\psi:=\begin{pmatrix}\psi_1 \\ \psi_2\end{pmatrix}:\mathbb{R}^2\to\mathbb{C}^2$ is  a spinor field, then the Dirac operator is
\begin{eqnarray}\label{z1}
\p\psi=\begin{pmatrix}0 & 1 \\ -1 & 0\end{pmatrix}\begin{pmatrix}\frac{\partial \psi_1}{\partial x} \\
\frac{\partial \psi_2}{\partial x}\end{pmatrix}+
\begin{pmatrix}0 &i \\ i & 0\end{pmatrix}\begin{pmatrix}\frac{\partial \psi_1}{\partial y} \\
\frac{\partial \psi_2}{\partial y}\end{pmatrix}=2\begin{pmatrix}\frac{\partial \psi_2}{\partial \overline{z}} \\
-\frac{\partial \psi_1}{\partial z}\end{pmatrix},
\end{eqnarray}
where
\[
\frac{\partial}{\partial z}=\frac{1}{2}(\frac{\partial}{\partial x}-i\frac{\partial}{\partial y}),
\quad \frac{\partial}{\partial \overline{z}}=\frac{1}{2}(\frac{\partial}{\partial x}+i\frac{\partial}{\partial y}).
\]
For more details on spin geometry and Dirac operators, one can refer to \cite{lawson1989spin,Bourguignon,Friedrich}.

We consider the following functional
\begin{eqnarray*}
L(\phi,\psi)&=&\jf{M}{\left(|d\phi|^2+\langle\psi,\slashed D\psi\rangle_{\Sigma M\otimes\phi^{*}TN}\right)}dvol\\
&=&
\jf{M}{\left(g_{ij}(\phi)h^{\alpha\beta}\frac{\partial\phi^i}{\partial x^\alpha}\frac{\partial\phi^j}{\partial x^\beta}+g_{ij}(\phi)\langle\psi^i,\slashed D\psi^j\rangle_{\Sigma M}\right)}dvol.
\end{eqnarray*}

The functional $L(\phi,\psi)$ is conformally invariant (see \cite{chen2006dirac}). That is, for any conformal diffeomorphism $f:M\rightarrow M$, setting
\[
\widetilde{\phi}=\phi\circ f \qquad and \qquad \widetilde{\psi}=\lambda^{-1/2}\psi\circ f,
\]
here $\lambda$ is the conformal factor of the conformal map $f$, $i.e.$ $f^*h=\lambda^2h$. Then $L(\widetilde{\phi},\widetilde{\psi})=L(\phi,\psi)$. Critical points $(\phi,\psi)$ of $L$ are called Dirac-harmonic maps from $M$ to $N$.

The Euler-Lagrange equations of the functional $L$ are
\begin{eqnarray}
\left(\Delta\phi^i+\Gamma^i_{jk}h^{\alpha\beta}\phi^j_\alpha\phi^k_\beta\right)\frac{\partial}{\partial y^i}(\phi(x))&=&R(\phi,\psi),\\
\slashed{D}\psi&=&0,
\end{eqnarray}
where $R(\phi,\psi)$ is defined by
\[
R(\phi,\psi)=\frac{1}{2}R^m_{lij}(\phi(x))\langle\psi^i,\nabla\phi^l\cdot\psi^j\rangle\frac{\partial}{\partial y^m}(\phi(x)).
\]
Here $R^m_{lij}$ stands for the Riemann curvature tensor of the target manifold $(N,g)$. One can refer to \cite{chen2005regularity, chen2006dirac}.

Using Nash's embedding theorem, we embed $N$ isometrically into some Euclidean space $\mathbb{R}^K$. Then, the critical points $(\phi,\psi)$ satisfy the Euler-Lagrange equations
\begin{eqnarray}
-\Delta \phi&=&A(\phi )(d\phi ,d\phi )-Re(P(\mathcal{A}(d\phi
(e_{\alpha }),e_{\alpha }\cdot \psi );\psi)), \label{dh5}
\\
\p\psi&=&\mathcal{A(}d\phi (e_{\alpha }),e_{\alpha }\cdot \psi
\mathcal{)}, \label{dh6}
\end{eqnarray}
where $\p$ is the usual Dirac operator, $A$ is the second fundamental form of $N$ in $\mathbb{R}^K$, and
\begin{align*}
\mathcal{A}(d\phi(e_\alpha),e_\alpha\cdot\psi)&:=(\nabla\phi^i\cdot\psi^j)\otimes A(\partial_{y^i},\partial_{y^j}),\\
Re(P(\mathcal{A}(d\phi
(e_{\alpha }),e_{\alpha }\cdot \psi );\psi))&:=P(A(\partial_{y^l},\partial_{y^j});\partial_{y^i})Re(\langle\psi^i,d\phi^l\cdot\psi^j\rangle)\in\mathbb{R}^K.
\end{align*}
Here $P(\xi;\cdot)$ denotes the shape operator satisfying $\langle P(\xi;X),Y\rangle=\langle A(X,Y),\xi\rangle$ for any $X,Y\in\Gamma(TN)$ and $Re(z)$ denotes the real part of $z\in \mathbb{C}$. We refer to \cite{chen2005regularity, chen2006dirac, Zhu, CJWZ2013, sharpzhu, jost-Liu-Zhu} for more details.

\

Before we state our main results, let us recall a definition of approximate Dirac-harmonic map in \cite{jost-Liu-Zhu-04}. Denote
\[
W^{2,2}(M,N):=\left\{\ \phi\in W^{2,2}(M,\mathbb{R}^K)\ with \ \phi(x)\in N\ for\ a.e.\ x\in M \ \right\},
\]
\begin{align*}
W^{1,4/3}(M,\Sigma M\otimes \phi^{*}TN):=\big\{\ &\psi\in W^{1,4/3}(M,\Sigma M\otimes \mathbb{R}^K)\ with \ \psi(x)\in \Sigma M\otimes \phi^{*}TN\\ &for\ a.e.\ x\in M \ \big\}.
\end{align*}

A pair of fields $(\phi,\psi)\in W^{2,2}(M,N)\times W^{1,\frac{4}{3}}(M,\Sigma M\times \phi^{*}TN)$ is called an approximate Dirac-harmonic map from $M$ to $N$ with boundary data $(\varphi(x),\chi(x))$, if there exists a pair of fields $(\tau(\phi,\psi), h(\phi,\psi))\in L^1(M)$ such that
\begin{eqnarray}
\tau(\phi,\psi)&=&\Delta\phi+
A(d\phi,d\phi)-Re\left( P(\mathcal{A}(d\phi(e_\alpha),e_\alpha\cdot\psi);\psi)\right),\label{dh1}\\
h(\phi,\psi)&=&\slashed{\partial}\psi-
\mathcal{A}(d\phi(e_\alpha),e_\alpha\cdot\psi),\label{dh2}
\end{eqnarray}
with the boundary data
\begin{eqnarray}\label{BOUND-DATA}
\begin{cases}
\phi(x)=\varphi(x),\quad &on\quad \partial M;\\
\mathcal{B}\psi(x)=\mathcal{B}\chi(x),\quad &on\quad \partial M,
\end{cases}
\end{eqnarray}
where $\mathcal{B}=\mathcal{B}^{\pm}$ is the chiral boundary operator defined by
\begin{align}\label{defi:-B}
\mathcal{B}^{\pm}: L^2(\partial M,\Sigma M\otimes \phi^{*}TN|_{\partial M}) & \rightarrow  L^2(\partial M,\Sigma M\otimes \phi^{*}TN|_{\partial M})   \\
\psi &\mapsto  \frac{1}{2} \left(Id \pm \overrightarrow{n}\cdot G \right)\cdot\psi,
\end{align}
where $\overrightarrow{n}$ is the outward unit normal vector field on $\partial M$, $G=i e_1\cdot e_2$ is the chiral operator defined using a local orthonormal
frame $\{e_\alpha\}_{\alpha=1}^{2}$ on $TM$ and satisfying:
\begin{align}\label{defi:-B-1}
G^2=Id,\quad G^*=G,\quad \nabla G=0,\quad G\cdot X=-X\cdot G,
\end{align}
for any $X\in\Gamma(TM)$. See e.g. \cite{CJWZ2013, CJSZ2014} for the notion of chiral boundary condition.

Therefore, $(\phi,\psi)$ is a Dirac-harmonic map if and only if $\tau(\phi,\psi)=h(\phi,\psi)=0$.

\

Dirac-harmonic maps were  introduced  in \cite{chen2005regularity, chen2006dirac}. They are  motivated by a model from quantum field theory, the supersymmetric sigma  model \cite{deligne1999quantum}. This subject generalizes the theory of harmonic maps and harmonic spinors. Similarly to the case of two dimensional harmonic maps, the conformal invariance of the energy functional $L$ leads to non-compactness of Dirac-harmonic maps in dimension 2 and hence one needs to study their blow-up theory, as in \cite{chen2005regularity,zhao2007energy, zhu1, Liulei}. For the blow-up theory of a more general model, whose critical points are called Dirac-harmonic maps with curvature terms, see \cite{jost-Liu-Zhu}. For approximate harmonic maps in dimension two, one can refer to e.g. \cite{SU, Jost1991, parker1996bubble, qing, DingWeiyueandTiangang,  qing1997bubbling, Lin-Wang, Lin-Wang2, LiJiayuandZhuXiangrong, Wang-Wei-Zhang,Wang} for the interior blow-up case and \cite{jost-Liu-Zhu-03,jost-Liu-Zhu-05, Huang-Wang} for the boundary blow-up cases under various boundary constraints.

In order to study the blow-up behavior of the Dirac-harmonic map flow from surfaces with boundary considered in \cite{CJSZ2014,jost-Liu-Zhu-02}, we introduced the notion of approximate Dirac-harmonic maps in \cite{jost-Liu-Zhu-04} and proved the energy identity and no neck result in the interior blow-up case for a sequence of such maps. In general, this sequence might blow up at a boundary point. In this paper, we shall consider the case that the sequence blows up at the boundary and hence complete the blow-up picture of the Dirac-harmonic map flow.

Denote the energy of $\phi$ on $\Omega\subset M$ by $$E(\phi;\Omega)=\int_\Omega|\nabla\phi|^2dM,$$
the energy of $\psi$ on $\Omega\subset M$ by $$E(\psi;\Omega)=\int_\Omega|\psi|^4dM,$$
and the energy of the pair $(\phi,\psi)$ on $\Omega\subset M$ by $$E(\phi,\psi;\Omega)=\int_\Omega(|\nabla\phi|^2+|\psi|^4)dM.$$
We shall often omit the domain $M$ from the notation and simply write $E(\phi)=E(\phi;M)$, $E(\psi)=E(\psi;M)$ and $E(\phi,\psi)=E(\phi,\psi;M)$.

\
Based on the interior blow-up results for approximate Dirac-harmonic maps studied in \cite{jost-Liu-Zhu-04}, we state our first main result in this paper concerning the boundary blow-up case:

\begin{thm}\label{thm:main-01}
Consider a sequence of approximate Dirac-harmonic maps $(\phi_n,\psi_n)\in C^{2}(M,N)\times C^{1}(M,\Sigma M\otimes \phi^{*}TN)$ from a compact Riemann surface $M$ with smooth boundary $\partial M$ to a compact Riemannian manifold $N$ satisfying
\[
E(\phi_n,\psi_n)+\|\tau(\phi_n,\psi_n)\|_{L^2}+\|h(\phi_n,\psi_n)\|_{L^4}\leq \Lambda,
\]
and with boundary data
\[
\phi_n|_{\partial M}=\varphi,\ \mathcal{B}\psi_n|_{\partial M}=\mathcal{B}\chi,
\]
where $\varphi\in C^{2+\alpha}(\partial M,N)$, $\chi\in C^{1+\alpha}(\partial M,\Sigma M\otimes \phi^{*}TN)$ for some $0<\alpha<1$. We assume $(\phi_n,\psi_n)\rightharpoonup (\phi,\psi)$ weakly in $W^{1,2}(M,N)\times L^{4}(M,\Sigma M\otimes\phi^*TN)$. Define the blow-up set
 \begin{equation}
\mathcal{S}:=\cap_{r>0}\big\{x\in M|\liminf_{n\to\infty}\int_{D(x,r)}(|d\phi_n|^2+|\psi_n|^4)\geq\overline{\epsilon}\big\},
\end{equation}
where $\overline{\epsilon}>0$ is the constant as in \eqref{def:small-energy}.
Then $\mathcal{S}$ is a (possibly empty) finite set  $\{p_1,...,p_q,...,p_I\}$, where $1\leq q\leq I$, $\{p_1,...,p_q\}\in M\setminus \partial M$, $\{p_{q+1},...,p_I\}\in \partial M$. Moreover, a subsequence, still denoted by $\{(\phi_k,\psi_k)\}$,  converges weakly in $W^{2,2}_{loc}(M\setminus \mathcal{S})\times W^{1,2}_{loc}(M\setminus \mathcal{S})$ to $(\phi,\psi)$ and for each $i=1,...,I$, there is a finite set of Dirac-harmonic spheres $(\sigma_i^l,\xi_i^l):S^2\to N$, $l=1,...,L_i$, such that
\begin{eqnarray}
\lim_{k\to\infty}E(\phi_k)&=&E(\phi)+\sum_{i=1}^I\sum_{l=1}^{L_i}E(\sigma_i^l),\\
\lim_{k\to\infty}E(\psi_k)&=&E(\psi)+\sum_{i=1}^I\sum_{l=1}^{L_i}E(\xi_i^l),
\end{eqnarray}
and the image $\phi(M\setminus \partial M)\cup\bigcup_{i=1}^q\bigcup_{l=1}^{L_i}(\sigma^l_i(S^2))$ is a connected set.
\end{thm}

\begin{rem}
In Theorem \ref{thm:main-01}, for those Dirac-harmonic spheres splitting off at the interior blow-up points, $i.e.$ $(\sigma_i^l,\xi_i^l):S^2\to N$, $i=1,...,q$; $l=1,...,L_i$, we know that the image of the map parts $\sigma_i^l$, $i=1,...,q$; $l=1,...,L_i$, are connected to the map part $\phi$ of the base field $(\phi,\psi)$ in the target manifold; this is proved in \cite{jost-Liu-Zhu-04},  the refined bubble tree can be constructed by applying similar arguments as in the harmonic map case given by \cite[Section 3]{Chen-Li} and \cite[Appendix]{Li-Wang2}. However, for those Dirac-harmonic spheres splitting off at the boundary blow-up points, $i.e.$ $(\sigma_i^l,\xi_i^l):S^2\to N$, $i=q+1,...,I$; $l=1,...,L_i$, it is not clear whether the images of the map parts $\sigma_i^l$, $i=q+1,...,I$; $l=1,...,L_i$ have the same property.
\end{rem}

To prove the energy quantization result near the boundary in Theorem \ref{thm:main-01}, we shall follow the general blow-up scheme developed for harmonic map type problems, however,  the proofs in this case are subtle and there are new difficulties arising when carrying out the neck analysis. Firstly, the method of the three circle type theorem used in the interior case in \cite{jost-Liu-Zhu-04} can not be applied to the boundary case and we need to apply certain integration argument to show the no neck energy property. Secondly, we need to establish a new Pohozaev type identity for approximate Dirac-harmonic maps  from surfaces with boundary (see Lemma \ref{lem:poho-bound}) which requires some algebraic property for the spinors,  see \eqref{algebraic-property}. Moreover, we need to derive a new Pohozaev type estimate (see Corollary \ref{cor:poho-bound}) by carrying out some exponential decay estimates, which are crucial in the proof of the above theorem. Finally, we would like to remark that the bubbling analysis at the boundary is more complicated than in the interior case and here we follow the scheme developed for approximate harmonic maps in \cite{jost-Liu-Zhu-03, jost-Liu-Zhu-05}.

\

With the help of Theorem \ref{thm:main-01}, we can now study the asymptotic behavior at infinite time for the Dirac-harmonic map flow in dimension 2.

The notion of Dirac-harmonic map flow was introduced in \cite{CJSZ2014}. In this flow,  one seeks a pair of fields $(\phi,\psi):M\times [0,\infty)\to N\times (\Sigma M\otimes\phi^{*}TN)$ that solves
\begin{eqnarray}\label{DHF1}
\begin{cases}
\partial_t \phi=\tau(\phi)-Re(P(\mathcal{A}(d\phi
(e_{\alpha }),e_{\alpha }\cdot \psi );\psi)),\quad &in\quad M\times (0,\infty);\\
\p\psi=\mathcal{A}(d\phi (e_{\alpha }),e_{\alpha }\cdot \psi
),\quad &in\quad M\times (0,\infty).
\end{cases}
\end{eqnarray}
with the following boundary-initial data:
\begin{eqnarray}\label{DHF1-BOUND-INITIAL}
\begin{cases}
\phi(x,t)=\varphi(x),\quad &on\quad \partial M\times [0,\infty);\\
\phi(x,0)=\phi_0(x),\quad &in\quad M;\\
\mathcal{B}\psi(x,t)=\mathcal{B}\chi(x),\quad &on\quad \partial M\times [0,\infty);\\
\phi_0(x)=\varphi(x),\quad &on\quad \partial M,
\end{cases}
\end{eqnarray}
where $\tau(\phi):=\Delta\phi+A(\phi)(\nabla\phi,\nabla\phi)$ is the tension field of $\phi$, $M$ is a compact spin Riemannian manifold with smooth boundary $\partial M$ and of dimension $dim\ M \geq 2$ and $\phi_0\in W^{1,2}(M,N)$, $\varphi\in C^{2+\alpha}(\partial M,N)$, $\chi\in C^{1+\alpha}(\partial M,\Sigma M\otimes \varphi^{*}TN)$ are given data. The short-time existence for the above flow \eqref{DHF1} \eqref{DHF1-BOUND-INITIAL}  was proved in \cite{CJSZ2014}.

When $M$ is a surface, it was shown in \cite{jost-Liu-Zhu-02} that there exists a unique global weak solution defined in $[0,\infty)\times M$ to \eqref{DHF1} with initial-boundary data \eqref{DHF1-BOUND-INITIAL} under some smallness assumption for $\|\phi_0\|_{H^{1}}+\|\mathcal{B}\chi\|_{L^{2}}$, which has at most finitely many singular times and enjoys the following property:
\begin{align}\label{z3}
E(\phi(t),\psi(t);M)+\int_{M^t}|\partial_t \phi|^2dxdt\leq C(M,E(\phi_0),\|\mathcal{B}\chi\|_{L^2(\partial M)}),  \quad \forall \
0\leq t<\infty.
\end{align}

It follows from \eqref{z3} that there exists a sequence $t_n\uparrow\infty$ such that $$(\phi_n,\psi_n):=(\phi(\cdot,t_n),\psi(\cdot,t_n))\in C^{2+\alpha}(M,N)\times C^{1+\alpha}(M,\Sigma M\times \phi^{*}TN)$$
is a sequence of approximate Dirac-harmonic maps with boundary-data $(\varphi,\chi)$ and satisfying $$h(\phi_n,\psi_n)=0$$ and $$\tau(\phi_n,\psi_n):=\partial_t\phi(\cdot,t_n)\ \ \text{with}    \ \ \|\tau(\phi_n,\psi_n)\|_{L^2}\to 0.$$

When such a flow blows up at infinite time and at interior points, it was proved in \cite{jost-Liu-Zhu-04} that an energy identity and no neck property hold during the blow-up process. In this paper, as an immediate corollary of Theorem \ref{thm:main-01} and as a complement of the blow-up picture at infinite time of such a flow given in \cite{jost-Liu-Zhu-04}, we obtain

\begin{thm}\label{thm:main-02}
Let $M$ be a compact spin Riemann surface with smooth boundary $\partial M$. Let $\phi_0\in H^1(M,N)$, $\varphi\in C^{2+\alpha}(\partial M,N)$, $\chi\in C^{1+\alpha}(\partial M,\Sigma M\otimes \varphi^{*}TN)$. Let $(\phi,\psi):M\times [0,\infty)\to N\times (\Sigma M\otimes\phi^{*}TN)$ be a global weak solution of \eqref{DHF1} and \eqref{DHF1-BOUND-INITIAL}, which has finitely many singular times and satisfies \eqref{z3}. Then there exist $t_n\uparrow\infty$, a Dirac-harmonic map $(\phi_\infty,\psi_\infty)\in C^{2+\alpha}(M,N)\times C^{1+\alpha}(M,\Sigma M\otimes \phi_\infty^{*}TN)$ with boundary data $\phi_\infty|_{\partial M}=\varphi$ and $\mathcal{B}\psi_\infty|_{\partial M}=\mathcal{B}\chi$, nonnegative integer $I$ and a possibly empty set with at most finitely many points $\{p_1,...,p_q,...,p_I\}\subset M$, where $1\leq q\leq I$, $\{p_1,...,p_q\}\in M\setminus \partial M$, $\{p_{q+1},...,p_I\}\in \partial M$ such that
\begin{itemize}
\item[(1)]  $(\phi_n,\psi_n):=(\phi(\cdot,t_n),\psi(\cdot,t_n))\rightharpoonup (\phi_\infty,\psi_\infty)\ in\ W^{1,2}(M,N)\times L^{4}(M,\Sigma M\times \phi_\infty^{*}TN)$;\\
\item[(2)]  $(\phi_n,\psi_n)\to (\phi_\infty,\psi_\infty)\ in\ W_{loc}^{1,2}(M\setminus \{p_1,...,p_I\})\times L_{loc}^{4}(M\setminus \{p_1,...,p_I\})$;\\
\item[(3)] For\ $1\leq i\leq I$,\mbox{ there exist a positive integer} $L_i$ and $L_i$ nontrivial Dirac-harmonic spheres $(\sigma_i^l,\xi_i^l):S^2\to N$, $i=1,...,I$; $l=1,...,L_i$ such that
\begin{eqnarray}
\lim_{n\to\infty}E(\phi_n)&=&E(\phi_\infty)+\sum_{i=1}^I\sum_{l=1}^{L_i}E(\sigma_i^l),\\
\lim_{n\to\infty}E(\psi_n)&=&E(\psi_\infty)+\sum_{i=1}^I\sum_{l=1}^{L_i}E(\xi_i^l).
\end{eqnarray}
and the image $\phi_\infty(M\setminus \partial M)\cup\bigcup_{i=1}^q\bigcup_{l=1}^{L_i}(\sigma^l_i(S^2))$ is a connected set.
\end{itemize}
\end{thm}

\

This paper is organized as follows. In Section 2, we  extend some basic lemmas to the boundary case, such as small energy regularity, Pohozaev's identity and removable singularity. Then, we recall some known results which will be used in this paper. In Section 3, we prove our main Theorem \ref{thm:main-01}.

\

\noindent\textbf{Notations:} We denote $\mathbb{R}^2_+=\{(x,y)\in\mathbb{R}^2|y\geq 0\}$, $D_r(x)=\{y\in\mathbb{R}^2||y-x|\leq r\}$, $D_r^+(x)=D_r(x)\cap\mathbb{R}^2_+$, $\partial^+D_r^+(x)=\partial D_r(x)\cap\mathbb{R}^2_+$, $\partial^0D_r^+(x)= D_r(x)\cap\partial \mathbb{R}^2_+$.

For simplicity, we also denote $D_r(0)$, $D^+_r(0)$, $D_1(0)$, $D_1^+(0)$ as $D_r$, $D^+_r$, $D$, $D^+$ respectively.

\

\section{Some basic lemmas}

\

In this section, we will prove some basic lemmas and recall some known results which will be used in this paper.

By standard elliptic theory, there exists a unique solution $u\in C^{2+\alpha}(M,\mathbb{R}^K)$ of
\begin{align*}
\begin{cases}
  \Delta u &= \quad 0,  \quad in\ M, \\
  u &=  \quad \varphi,  \quad on \ \partial M,
  \end{cases}
\end{align*}
satisfying $$\|u\|_{C^{2+\alpha}(M)}\leq C(\alpha,M)\|\varphi\|_{C^{2+\alpha}(\partial M)}.$$
For simplicity of notation,  in the sequel, we will also denote this solution as $\varphi$.

\

Firstly, we prove a small energy regularity theorem for the boundary case. For similar results for approximate harmonic maps, one can refer to the main estimate 3.2 in \cite{SU} and Lemma 2.1 in \cite{DingWeiyueandTiangang} for the interior case and one can also refer to Lemma 4.1 in \cite{jost-Liu-Zhu-03}, Lemma 2.4 in \cite{jost-Liu-Zhu-05} for various boundary cases.
\begin{thm}\label{smthm-bound}
There is a small constant $\epsilon_0>0$ depending only on $p,q$ and $N$, such that if $(\phi,\psi)\in W^{2,p}(D^+,N)\times W^{1,q}(D^+,\Sigma D^+\otimes \phi^{*}TN)$ is an approximate Dirac-harmonic map from the upper unit disc $D^+ \subset\mathbb{R}^2$ to a compact Riemannian manifold $(N,g)\subset \mathbb{R}^K$ with $\tau(\phi,\psi)\in L^p$ and $h(\phi,\psi)\in L^q$ for some $1< p\leq 2$ and some $\frac{4}{3}< q\leq 2$, and with boundary data \eqref{BOUND-DATA},  satisfying
\begin{eqnarray}
E(\phi,\psi;D^+)=\jf{D^+}{(|d\phi|^2+|\psi|^4)}dx<(\epsilon_0)^2,
\end{eqnarray}
then
\begin{eqnarray*}
\|\phi-\overline{\varphi}\|_{W^{2,p}(D^+_\frac{1}{2})}&\leq& C(\|d\phi\|_{L^p(D^+)}+\|\psi\|_{L^{2p}(D^+)}+\||\tau|\|_{L^{p}(D^+)} +\|\nabla\varphi\|_{W^{1,p}(D^+)})
,\\
 \|\psi\|_{W^{1,q}(D^+_\frac{1}{2})}&\leq& C(\|\psi\|_{L^q(D^+)}+\|h\|_{L^q(D^+)}+\|\mathcal{B}\chi\|_{W^{1-1/q,q}(\partial^0 D^+)}),
\end{eqnarray*}
where $\overline{\varphi}:=\int_{\partial^0 D^+_{1/2}}\varphi\in \mathbb{R}^K $ and $C>0$ is a constant depending only on $p,\ q,\ N,\|\varphi\|_{C^2},\|\chi\|_{C^1}$.

Moreover, by the Sobolev embedding $W^{2,p}(\mathbb{R}^2)\subset C^0(\mathbb{R}^2)$, we have
\begin{align}
\|\phi\|_{Osc(D^+_{1/2})}=\sup_{x,y\in D^+_{1/2}}|\phi(x)-\phi(y)|\leq C(\|\nabla \phi\|_{L^2(D^+)}+\|\tau(u)\|_{L^p(D^+)}+\|\nabla\varphi\|_{W^{1,p}(D^+)}).
\end{align}
\end{thm}
\begin{proof}
Without loss of generality, we assume $\int_{\partial^0 D^+_{\frac{1}{2}}}\varphi =0$.

Choosing a cut-off function $\eta\in C_0^\infty(D^+)$ satisfying $0\leq\eta\leq1,\eta|_{D^+_{3/4}}\equiv1,|\nabla\eta|+|\nabla^2\eta|\leq C$, by standard theory of first order elliptic equation, for any $1<q<2$, we have
\begin{align*}
\|\eta\psi\|_{W^{1,q}(D^+)}&\leq C(\|\slashed{\partial}(\eta\psi)\|_{L^q(D^+)}+\|\mathcal{B}\psi\|_{W^{1-1/q,q}(\partial^0 D^+)})\\
&\leq
C(\|\nabla\eta\cdot\psi+\eta\slashed{\partial}\psi\|_{L^q(D^+)}
+\|\mathcal{B}\psi\|_{W^{1-1/q,q}(\partial^0 D^+)})\\
&\leq
C\left(\|\psi\|_{L^q(D^+)}+\||d\phi||\eta\psi|\|_{L^q(D^+)}+\|h\|_{L^q(D^+)}
+\|\mathcal{B}\psi\|_{W^{1-1/q,q}(\partial^0 D^+)}\right)\\
&\leq
C\|d\phi\|_{L^2(D^+)}\|\eta\psi\|_{L^{\frac{2q}{2-q}}(D^+)}
+C(\|\psi\|_{L^q(D^+)}+\|h\|_{L^q(D^+)}+\|\mathcal{B}\psi\|_{W^{1-1/q,q}(\partial^0 D^+)})\\
&\leq
C\epsilon_0\|\eta\psi\|_{L^{\frac{2q}{2-q}}(D^+)}+C(\|\psi\|_{L^q(D^+)}+\|h\|_{L^q(D^+)}
+\|\mathcal{B}\psi\|_{W^{1-1/q,q}(\partial^0 D^+)}).
\end{align*}
Taking $\epsilon_0>0$ sufficiently small, by Sobolev embedding, we get
\begin{align}\label{inequality:07}
\|\eta\psi\|_{L^{\frac{2q}{2-q}}(D^+)}\leq \|\eta\psi\|_{W^{1,q}(D^+)}\leq C(\|\psi\|_{L^q(D^+)}+\|h\|_{L^q(D^+)}+\|\mathcal{B}\psi\|_{W^{1-1/q,q}(\partial^0 D^+)}).
\end{align}

Computing directly, we obtain
\begin{align*}
|\Delta(\eta\phi)|&=|\eta\Delta\phi+2\nabla\eta\nabla\phi+\phi\Delta\eta|\notag\\
&\leq
C\left(|\phi|+|d\phi|+|d\phi||\eta d\phi|+|\psi|^2|\eta d\phi|+|\tau|\right)\notag\\
&\leq
C(|d\phi|+|\psi|^2)|d(\eta\phi)|+
C\left(|\phi|+|d\phi|+|\psi|^2+|\tau|\right).
\end{align*}
By standard elliptic estimates and Poincar\'{e}'s inequality, for any $1<p<2$, we have
\begin{align*}
\|\eta\phi\|_{W^{2,p}(D^+)}&\leq C\|(|d\phi|+|\psi|^2)|d(\eta\phi)|\|_{L^{p}(D^+)}
+C(\|d\phi\|_{L^{p}(D^+)}+\||\psi|^2\|_{L^{p}(D^+)}\\&\quad+\||\tau|\|_{L^{p}(D^+)}
+\|\varphi\|_{W^{2,p}(D^+)})\\
&\leq
C\|d(\eta\phi)\|_{L^{\frac{2p}{2-p}}(D^+)}(\|d\phi\|_{L^2(D^+)}+\|\psi\|^2_{L^4(D^+)})
+C(\|d\phi\|_{L^{p}(D^+)}\\&\quad+\|\psi\|^2_{L^{2p}(D^+)}
+\||\tau|\|_{L^{p}(D^+)}+\|\varphi\|_{W^{2,p}(D^+)})\\
&\leq
C\epsilon_0\|d(\eta\phi)\|_{L^{\frac{2p}{2-p}}(D^+)}+C(\|d\phi\|_{L^p(D^+)}+\|\psi\|^2_{L^{2p}(D^+)}+\||\tau|\|_{L^{p}(D^+)}
\\&\quad+\|\nabla\varphi\|_{W^{1,p}(D^+)}).
\end{align*}
Taking $\epsilon_0>0$ sufficiently small, we have
\begin{align}\label{inequality:01}
\|\nabla(\eta\phi)\|_{L^{\frac{2p}{2-p}}(D^+)}&\leq C\|\eta\phi\|_{W^{2,p}(D^+)}\notag\\&\leq
C(\|d\phi\|_{L^p(D^+)}+\|\psi\|^2_{L^{2p}(D^+)}+\||\tau|\|_{L^{p}(D^+)} +\|\nabla\varphi\|_{W^{1,p}(D^+)}).
\end{align}
So, we have proved the theorem in the case $1<p<2$, $4/3<q<2$.

For the case $p=2$, $4/3<q<2$, taking $p=\frac{q}{2(q-1)}\in (1,2)$ and $p=\frac{4}{3}$ in \eqref{inequality:01}, by Sobolev embedding, we have
\begin{align}
&\|\nabla\phi\|_{L^{4}(D^+_{3/4})}+\|\nabla\phi\|_{L^{\frac{2q}{3q-4}}(D^+_{3/4})}\notag\\
&\leq
C(\|d\phi\|_{L^2(D^+)}+\|\psi\|^2_{L^{4}(D^+)}+\||\tau|\|_{L^{2}(D^+)} +\|\nabla\varphi\|_{W^{1,2}(D^+)}).
\end{align}
By \eqref{inequality:07} and the $W^{2,2}$-estimate for the Laplace operators, we obtain
\begin{align*}
\|\phi\|_{W^{2,2}(D^+_{1/2})}&\leq C(\|\Delta \phi\|_{L^2(D^+_{3/4})}+\|\nabla\phi\|_{L^{2}(D^+)}+\|\nabla\varphi\|_{W^{1,2}(D^+)})\\
&\leq C(\|\nabla \phi\|^2_{L^4(D^+_{3/4})}+\|\nabla\phi\|_{\frac{2q}{3q-4}(D^+_{3/4})}\||\psi|^2\|_{L^{\frac{q}{2-q}}(D^+_{3/4})} +\|\nabla\phi\|_{L^{2}(D^+)}\\&\quad+\|\nabla\varphi\|_{W^{1,2}(D^+)})\\
&\leq C(\|d\phi\|_{L^2(D^+)}+\|\psi\|_{L^{4}(D^+)}+\||\tau|\|_{L^{2}(D^+)} +\|\nabla\varphi\|_{W^{1,2}(D^+)}).
\end{align*}

For the case $q=2$, $1<p<2$, taking $q=\frac{2p}{3p-2}\in (1,2)$ in \eqref{inequality:07}, we get
\begin{align}
\|\psi\|_{L^{\frac{p}{p-1}}(D^+_{3/4})}\leq C(\|\psi\|_{L^2(D^+)}+\|h\|_{L^2(D^+)}+\|\mathcal{B}\psi\|_{W^{1-1/2,2}(\partial^0 D^+)}).
\end{align}
By \eqref{inequality:01} and $W^{1,2}$-estimates for the  Dirac operator, we arrive at
\begin{align*}
\|\psi\|_{W^{1,2}(D^+_{1/2})}&\leq C(\|\p\psi\|_{L^2(D^+_{3/4})}+\|\psi\|_{L^4(D_{3/4})} +\|\mathcal{B}\chi\|_{W^{1-1/2,2}(\partial^0D^+)})\\
&\leq
C(\|\nabla\phi\|_{L^{\frac{2p}{2-p}}(D^+_{3/4})}\|\psi\|_{L^{\frac{p}{p-1}}(D^+_{3/4})} +\|\psi\|_{L^4(D_{3/4})} +\|\mathcal{B}\chi\|_{W^{1-1/2,2}(\partial^0D^+)})\\
&\leq C(\|\psi\|_{L^2(D^+)}+\|h\|_{L^2(D^+)}+\|\mathcal{B}\psi\|_{W^{1-1/2,2}(\partial^0 D^+)}).
\end{align*}

For the case $p=q=2$, taking $q=\frac{8}{5}$ in \eqref{inequality:07} and $p=\frac{4}{3}$ in \eqref{inequality:01}, we will obtain a $L^{8}(D^+_{3/4})$-bound for $\psi$ and a $L^4(D^+_{3/4})$ bound for $\nabla\phi$. Then one can apply the $W^{2,2}$-boundary estimate for the Laplace operator and the $W^{1,2}$-boundary estimate for the Dirac operator to get the conclusion of the theorem.
\end{proof}

\

Next we shall derive a Pohozaev type identity for approximate Dirac-harmonic maps with boundary data, extending the interior case given in Lemma 2.3 in \cite{jost-Liu-Zhu-04}. For corresponding results for two dimensional approximate harmonic maps, one can refer to Lemma 2.4 \cite{Lin-Wang} for the interior case and refer to Lemma 4.3 in \cite{jost-Liu-Zhu-03} and Lemma 2.5 in \cite{jost-Liu-Zhu-05}  for various boundary cases.

\begin{lem}\label{lem:poho-bound} {\rm (Pohozaev type identity)}
Let $\Omega\subset \mathbb{R}^2$ be a bounded smooth domain. If  $D^+\subset\Omega\subset\mathbb{R}_+^2$ and $(\phi,\psi)\in C^2(\Omega,N)\times C^1(\Omega,\Sigma \Omega\otimes \phi^*TN)$ is an approximate Dirac-harmonic map with boundary data \eqref{BOUND-DATA} on $\partial^0\Omega$, then for any $0<t<\frac{1}{2}$, we have
\begin{align}\label{equation:14}
t\int_{\partial^+ D^+_t}(|\phi_r|^2-\frac{1}{2}|\nabla\phi|^2)&=\frac{1}{2}\int_{\partial^+ D^+_t}\langle \psi,r^{-1}\partial_\theta\cdot\psi_\theta\rangle-\frac{1}{2}\int_{D^+_t}\langle \psi,\slashed D\psi\rangle dx-Re\int_{D^+_t}\langle \slashed D\psi,r\psi_r\rangle dx\notag\\&\quad+\int_{D^+_t}r(\phi-\varphi)_r\tau dx+\frac{1}{2}\int_{\partial^0 D^+_t}\langle \psi,\frac{\partial}{\partial x^2}\cdot r\psi_r\rangle\notag\\&\quad+\int_{\partial^+ D^+_t}r\phi_r\varphi_r-\int_{D^+_t}\nabla\phi (\nabla \varphi+r\nabla \varphi_r)dx\notag\\
&\quad+\int_{D^+_t}\langle r\varphi_r,A(\phi)(d\phi,d\phi)-Re\left( P(\mathcal{A}(d\phi(e_\alpha),e_\alpha\cdot\psi);\psi)\right)\rangle dx,
\end{align}
where $(r,\theta)$ are  polar coordinates in $D$ centered at $0$, $\phi_r=\frac{\partial\phi}{\partial r}$, $\psi_r=\widetilde{\nabla}_{\frac{\partial}{\partial r}}\psi$ and $\psi_\theta=\widetilde{\nabla}_{\frac{\partial}{\partial \theta}}\psi$.
\end{lem}

Before we prove this lemma, let us recall two basic facts for Dirac operators and spinors with chiral boundary constraint,

\

\textbf{Fact 1:}
For any $\psi,\omega\in W^{1,3/4}(M,\Sigma M\otimes \phi^{*}TN)$ satisfying
$$\mathcal{B}\psi|_{\partial M}=\mathcal{B}\omega|_{\partial M}=0,$$
we have
\begin{equation}
\langle\overrightarrow{n}\cdot\psi,\omega\rangle=0 \ on \ \partial M,
\end{equation}
where $\overrightarrow{n}$ is the unit normal vector field on $\partial M$.

For a proof of this straightforward fact, see e.g. [\cite{CJWZ2013}, Prop 3.1].

\

\textbf{Fact 2:}
For any $\psi,\omega\in W^{1,3/4}(M,\Sigma M\otimes \phi^{*}TN)$, we have
\begin{equation}
\int_M\langle\psi,\slashed{D}\omega\rangle dx=\int_M\langle\slashed{D}\psi,\omega\rangle dx
-\int_{\partial M}\langle\overrightarrow{n}\cdot\psi,\omega\rangle
\end{equation}
where $\langle\psi,\omega\rangle:=h_{ij}\langle\psi^i,\omega^j\rangle$.

For a proof of this well-known fact, see e.g. [\cite{CJWZ2013}, Prop 3.2].

\

\noindent{\it Proof of Lemma \ref{lem:poho-bound}}:
Multiplying the equation \eqref{dh1} by $r(\phi-\varphi)_r$ and integrating over $D^+_t$, noting the fact that $r\partial_r\phi=x^\beta\partial_\beta\phi$ and recalling Proposition 2.2 in \cite{jost-Liu-Zhu-04} that
\begin{align} \label{algebraic-property}
\langle\psi,\widetilde{\nabla}_{\frac{\partial}{\partial x^\beta}}(\slashed D\psi)=
2\langle Re\left( P(\mathcal{A}(d\phi(e_\alpha),e_\alpha\cdot\psi);\psi)\right),\nabla_{\frac{\partial}{\partial x^\beta}}\phi\rangle+\langle\psi,\slashed D\widetilde{\nabla}_{\frac{\partial}{\partial x^\beta}}\psi\rangle,
\end{align} we get
\begin{align*}
\int_{D^+_t}r(\phi-\varphi)_r\tau dx&=\int_{D^+_t}r(\phi-\varphi)_r\Delta\phi dx-\int_{D^+_t}\langle r\phi_r,Re\left( P(\mathcal{A}(d\phi(e_\alpha),e_\alpha\cdot\psi);\psi)\right)\rangle dx\\
&\quad-\int_{D^+_t}\langle r\varphi_r,A(\phi)(d\phi,d\phi)-Re\left( P(\mathcal{A}(d\phi(e_\alpha),e_\alpha\cdot\psi);\psi)\right)\rangle dx\\
&=\int_{D^+_t}r(\phi-\varphi)_r\Delta\phi dx+\frac{1}{2}\int_{D^+_t}\langle x^\beta\psi,\slashed D\psi_\beta\rangle dx-\frac{1}{2}\int_{D^+_t}\langle x^\beta\psi,\widetilde{\nabla}_{\partial_\beta}\slashed D\psi\rangle dx\\
&\quad-\int_{D^+_t}\langle r\varphi_r,A(\phi)(d\phi,d\phi)-Re\left( P(\mathcal{A}(d\phi(e_\alpha),e_\alpha\cdot\psi);\psi)\right)\rangle dx\\
:&=\mathbb{I}+\mathbb{II}+\mathbb{III}+\mathbb{IV}.
\end{align*}

On one hand, by integrating by parts, we have
\begin{align*}
\mathbb{I}&=\int_{\partial^+ D^+_t}r|\phi_r|^2-\int_{\partial^+ D^+_t}r\phi_r\varphi_r-\int_{D^+_t}\nabla\phi\nabla(r(\phi-\varphi)_r)dx\\
&=
\int_{\partial^+ D^+_t}r|\phi_r|^2-\int_{\partial^+ D^+_t}r\phi_r\varphi_r-\int_{D^+_t}\nabla\phi \nabla (\phi-\varphi)dx-\frac{1}{2}\int_{D^+_t}r\partial_r|\nabla\phi|^2dx\\
&\quad+\int_{D^+_t}r\nabla\phi\nabla\varphi_rdx
\\
&=
t\int_{\partial^+ D^+_t}(|\phi_r|^2-\frac{1}{2}|\nabla\phi|^2)-\int_{\partial^+ D^+_t}r\phi_r\varphi_r+\int_{D^+_t}\nabla\phi (\nabla \varphi+r\nabla \varphi_r)dx,
\end{align*}
where the last equality follows from the fact that
\begin{align*}
-\frac{1}{2}\int_{D^+_t}r\partial_r|\nabla\phi|^2dx&=-\frac{1}{2}\int_{\partial^+ D^+_1}\int_0^t r^2\partial_r|\nabla\phi|^2 drd\theta\\
&=-\frac{1}{2}\int_{\partial^+ D^+_t} t|\nabla\phi|^2 +\int_{ D^+_t} |\nabla\phi|^2dx.
\end{align*}

On the other hand, by \textbf{Fact 2} , we get
\begin{align}\label{equation:06}
2\mathbb{II}&=\int_{D_t^+}\langle x^\beta\psi,\slashed D\psi_\beta\rangle dx\notag\\
&=\int_{D_t^+}\langle \slashed D(x^\beta\psi),\psi_\beta\rangle dx-\int_{\partial^+ D_t^+}\langle \frac{\partial}{\partial r}\cdot x^\beta\psi,\psi_\beta\rangle +\int_{\partial^0 D_t^+}\langle \frac{\partial}{\partial x^2}\cdot x^\beta\psi,\psi_\beta\rangle \notag\\
&=
-\int_{D_t^+}\langle \psi,\slashed D\psi\rangle dx+\int_{D_t^+}\langle \slashed D\psi,r\psi_r\rangle dx+\int_{\partial^+ D_t^+}\langle \psi,r\partial_r\cdot\psi_r\rangle-\int_{\partial^0 D_t^+}\langle \psi,\frac{\partial}{\partial x^2}\cdot r\psi_r\rangle.
\end{align}
Integrating by parts, it follows that
\begin{align}\label{equation:07}
2\mathbb{III}&=-\int_{D_t^+}\langle x^\beta\psi,\widetilde{\nabla}_{\partial_\beta}\slashed D\psi\rangle dx\notag\\
&=-\int_{\partial^+ D_t^+}\langle r\psi,\slashed D\psi\rangle dx+\int_{D_t^+}\langle\widetilde{\nabla}_{\partial_\beta} (x^\beta\psi),\slashed D\psi\rangle dx\notag\\
&=2\int_{D_t^+}\langle\psi,\slashed D\psi\rangle dx+\int_{D_t^+}\langle r\psi_r,\slashed D\psi\rangle dx-\int_{\partial^+ D_t^+}\langle r\psi,\slashed D\psi\rangle .
\end{align}

Thus, we have
\begin{align*}
\mathbb{II}+\mathbb{III}
&=
\frac{1}{2}\int_{D_t^+}\langle \psi,\slashed D\psi\rangle dx+Re\int_{D_t^+}\langle \slashed D\psi,r\psi_r\rangle dx-\frac{1}{2}\int_{\partial^+ D^+_t}\langle \psi,r^{-1}\partial_\theta\cdot\psi_\theta\rangle\\
&\quad-\frac{1}{2}\int_{\partial^0 D^+_t}\langle \psi,\frac{\partial}{\partial x^2}\cdot r\psi_r\rangle .
\end{align*}
Combining these estimates, we get \eqref{equation:14}. This finishes the proof of the lemma.       \hskip20pt     $\Box$

\

As a consequence of Lemma \ref{lem:poho-bound}, we derive the following Pohozaev type estimate, which plays a key role in the proof of Theorem \ref{thm:main-01}.

\begin{cor}\label{cor:poho-bound}{\rm (Pohozaev type estimate)}
Under the assumption of Lemma \ref{lem:poho-bound}, if $$E(\phi,\psi;D^+)+\|\tau(\phi,\psi)\|_{L^2}+\|h(\phi,\psi)\|_{L^4}\leq \Lambda,$$ then for any $0<t<\frac{1}{4}$ and $0<\varepsilon<\frac{1}{4}$, we have
\begin{align}
\int_{ D^+_{2t}\setminus D^+_t}(|\phi_r|^2-\frac{1}{2}|\nabla\phi|^2)dx
\leq&
 \varepsilon\int_{D^+_{2t}\setminus D^+_t}|r^{-1}\frac{\partial\phi}{\partial \theta}|^{2}dx +\frac{C}{\varepsilon}\int_{D^+_{2t}\setminus D^+_t}|\psi|^{4}dx\notag\\&+C\int_{D^+_{2t}\setminus D^+_t}|r^{-1}\frac{\partial\psi}{\partial \theta}|^{\frac{4}{3}}dx+C\sqrt{t},
\end{align}
where $C$ is a positive constant depending only on $\Lambda, \ N,\   \|\varphi\|_{C^2},\ \|\chi\|_{C^1}$.
\end{cor}
\begin{proof}
Firstly, by equation \eqref{dh2} and elliptic theory, we have $$\|\psi\|_{W^{1,\frac{4}{3}}(D^+_{\frac{1}{2}})}\leq C(\|\nabla\phi\|_{L^2(D^+)}\|\psi\|_{L^4(D^+)}+\|h\|_{L^{\frac{4}{3}}(D^+)}+\|\mathcal{B}\chi\|_{W^{1/4,4/3}(\partial^0 D^+)})\leq C.$$

Thanks to Lemma \ref{lem:poho-bound}, for any $0<t<\frac{1}{2}$, we have
\begin{align}\label{equation:12}
t\int_{\partial^+ D^+_t}(|\phi_r|^2-\frac{1}{2}|\nabla\phi|^2):=\mathbb{I}_1+...+\mathbb{I}_8.
\end{align}
Using Young's inequality and the fact that
\begin{equation}\label{equation:03}
\psi_r=\widetilde{\nabla}_{\frac{\partial}{\partial r}}\psi=\frac{\partial\psi}{\partial r}+\psi^i\otimes A(d\phi(\frac{\partial}{\partial r}),\frac{\partial}{\partial y^i}),\end{equation} where $\frac{\partial\psi}{\partial r}=(\frac{\partial\psi^1}{\partial r},...,\frac{\partial\psi^K}{\partial r})$, we obtain
\begin{align}\label{equation:05}
\mathbb{I}_2+\mathbb{I}_3+\mathbb{I}_4&\leq Ct(\|\psi\|_{L^4(D^+_t)}+\|\psi_r\|_{L^{\frac{4}{3}}(D^+_t)})\|h\|_{L^4(D^+_t)} \notag\\&\quad+Ct(\|\nabla\phi\|_{L^2(D^+_t)}+\|\nabla\varphi\|_{L^2(D^+_t)})\|\tau\|_{L^2(D^+_t)} \notag\\&
\leq Ct(\|\psi\|_{L^4(D^+_t)}+\|\nabla\psi\|_{L^{\frac{4}{3}}(D^+_t)} +\|\nabla\phi\|_{L^{2}(D^+_t)}\|\psi\|_{L^{4}(D^+_t)})\|h\|_{L^4(D^+_t)} + Ct\notag\\&\leq Ct.
\end{align}

As for $\mathbb{I}_5$, we have
\begin{align}\label{equation:08}
\mathbb{I}_5&=\frac{1}{2}\int_{\partial^0 D^+_t}\langle \psi,\frac{\partial}{\partial x^2}\cdot r\psi_r\rangle\notag\\
&=\frac{1}{2}\int_{\partial^0 D^+_t}\langle \psi-\chi,\frac{\partial}{\partial x^2}\cdot r(\psi-\chi)_r\rangle+\frac{1}{2}\int_{\partial^0 D^+_r}\langle \chi,\frac{\partial}{\partial x^2}\cdot r\psi_r\rangle\notag\\
&\quad+\frac{1}{2}\int_{\partial^0 D^+_t}\langle \psi,\frac{\partial}{\partial x^2}\cdot r\chi_r\rangle-\frac{1}{2}\int_{\partial^0 D^+_t}\langle \chi,\frac{\partial}{\partial x^2}\cdot r\chi_r\rangle\notag\\
&=\frac{1}{2}\int_{\partial^0 D^+_t}\langle \chi,\frac{\partial}{\partial x^2}\cdot r\psi_r\rangle+\frac{1}{2}\int_{\partial^0 D^+_t}\langle \psi,\frac{\partial}{\partial x^2}\cdot r\chi_r\rangle-\frac{1}{2}\int_{\partial^0 D^+_t}\langle \chi,\frac{\partial}{\partial x^2}\cdot r\chi_r\rangle,
\end{align}
where the last equality follows from \textbf{Fact 1}  which tells us that $$\frac{1}{2}\int_{\partial^0 D^+_t}\langle \psi-\chi,\frac{\partial}{\partial x^2}\cdot r(\psi-\chi)_r\rangle=0.$$

Computing directly, we get
\begin{align*}
\frac{1}{2}\int_{\partial^0 D^+_t}\langle \chi,\frac{\partial}{\partial x^2}\cdot r\psi_r\rangle&=-\frac{1}{2}\int_{-t}^t\langle x^1\frac{\partial}{\partial x^2}\cdot \chi,\widetilde{\nabla}_{\frac{\partial}{\partial x^1}}\psi\rangle dx^1\\
&=-\frac{1}{2}\int_{-t}^t\frac{\partial}{\partial x^1}\langle x^1\frac{\partial}{\partial x^2}\cdot \chi,\psi\rangle dx^1+\frac{1}{2}\int_{-t}^t\langle\widetilde{\nabla}_{\frac{\partial}{\partial x^1}} (x^1\frac{\partial}{\partial x^2}\cdot \chi),\psi\rangle dx^1\\
&\leq
Ct(|\psi|(t,0)+|\psi|(-t,0))+C\sqrt{t}\|\psi\|_{L^2(\partial^0 D_{\frac{1}{2}}^+)}.
\end{align*}

By H\"{o}lder's inequality and trace theory, we have
\begin{align*}
\frac{1}{2}\int_{\partial^0 D^+_t}\langle \psi,\frac{\partial}{\partial x^2}\cdot r\chi_r\rangle-\frac{1}{2}\int_{\partial^0 D^+_t}\langle \chi,\frac{\partial}{\partial x^2}\cdot r\chi_r\rangle&\leq C(\sqrt{t}\|\psi\|_{L^2(\partial^0 D_{\frac{1}{2}}^+)}+t)\\
&\leq C(\sqrt{t}\|\psi\|_{W^{1,\frac{4}{3}}(D^+_{\frac{1}{2}})}+t),
\end{align*}
where $C$ is a constant depending only on $\|\chi\|_{C^1}$.

Then \eqref{equation:08} implies
\begin{align}\label{equation:09}
\mathbb{I}_5\leq
Ct(|\psi|(t,0)+|\psi|(-t,0))+C\sqrt{t}\|\psi\|_{L^2(\partial^0 D_{\frac{1}{2}}^+)}+Ct.
\end{align}

For $ \mathbb{I}_7$ and $\mathbb{I}_8$, it is easy to see that
\begin{align}\label{equation:11}
\mathbb{I}_7+ \mathbb{I}_8\leq Ct.
\end{align}
Multiplying \eqref{equation:12} by $\frac{1}{t}$ and integrating from $t$ to $2t$, we get
\begin{align*}
&\int_{ D^+_{2t}\setminus D^+_t}(|\phi_r|^2-\frac{1}{2}|\nabla\phi|^2)dx\notag\\
&\leq
\frac{1}{2}\int_t^{2t}\frac{1}{r}\int_{\partial^+ D^+_r}\langle \psi,r^{-1}\partial_\theta\cdot\psi_\theta\rangle d\theta dr+\int_t^{2t}\frac{1}{r}\int_{\partial^+ D^+_r}r\phi_r\varphi_rd\theta dr\\
&\quad+C\int_t^{2t}(|\psi|(r,0)+|\psi|(-r,0)+1+\frac{1}{\sqrt{r}}) dr\\
&\leq
C\|\psi\|_{L^4(D^+_{2t}\setminus D^+_t)}\|r^{-1}\psi_\theta\|_{L^{\frac{4}{3}}(D^+_{2t}\setminus D^+_t)}+Ct(1+\|d\phi\|_{L^2(D^+_{2t}\setminus D^+_t)})+C\sqrt{t}(1+\|\psi\|_{L^2(\partial^0D_{\frac{1}{2}}^+)})\\
&\leq \varepsilon\int_{D^+_{2t}\setminus D^+_t}|r^{-1}\frac{\partial\phi}{\partial \theta}|^{2}dx +\frac{C}{\varepsilon}\int_{D^+_{2t}\setminus D^+_t}|\psi|^{4}dx+C\int_{D^+_{2t}\setminus D^+_t}|r^{-1}\frac{\partial\psi}{\partial \theta}|^{\frac{4}{3}}dx+C\sqrt{t},
\end{align*}
where the last inequality follows from Young's inequality, the trace theory
$$\|\psi\|_{L^2(\partial^0D_{\frac{1}{2}}^+)}\leq C\|\psi\|_{W^{1,\frac{4}{3}}(D_{\frac{1}{2}}^+)}$$ and the fact
\begin{equation*}
\psi_\theta=\widetilde{\nabla}_{\frac{\partial}{\partial \theta}}\psi=\frac{\partial\psi}{\partial \theta}+\psi^i\otimes A(d\phi(\frac{\partial}{\partial \theta}),\frac{\partial}{\partial y^i}).
\end{equation*}
This finishes the proof.
\end{proof}

\

In the end of this section, we recall some known results for (approximate) Dirac-harmonic maps which are used in this paper.

\

\begin{thm}[Theorem 2.1., \cite{jost-Liu-Zhu-04}]\label{smthm}
There is a small constant $\epsilon'_0>0$ depending on $p,q$ and $N$, such that if $(\phi,\psi)\in W^{2,p}(D,N)\times W^{1,q}(D,\Sigma D\otimes \phi^{*}TN)$ is an approximate Dirac-harmonic map from the unit disc $D$ in $\mathbb{R}^2$ to a compact Riemannian manifold $(N,g)$ with $\tau (\phi,\psi)\in L^p$ and $h(\phi,\psi)\in L^q$ for some $\frac{4}{3}\leq p\leq 2$ and some $\frac{8}{5}\leq q\leq 2$, and satisfies
\begin{eqnarray}
E(\phi,\psi;D)=\jf{D}{(|d\phi|^2+|\psi|^4)}dx<(\epsilon_0')^2,
\end{eqnarray}
then
\begin{eqnarray*}
\|\phi-\overline{\phi}\|_{W^{2,p}(D_\frac{1}{2})}&\leq& C(\|d\phi\|_{L^2(D)}+\|\tau\|_{L^p(D)}),\\
 \|\psi\|_{W^{1,q}(D_\frac{1}{2})}&\leq& C(\|\psi\|_{L^4(D)}+\|h\|_{L^q(D)}),
\end{eqnarray*}
where $\overline{\phi}:=\frac{1}{|D_{1/2}|}\int_{D_{1/2}}\phi dx$ and $C>0$ is a constant depending only on $p,\ q,\ N$.

Moreover, by the Sobolev embedding $W^{2,p}(\mathbb{R}^2)\subset C^0(\mathbb{R}^2)$, we have
\begin{equation}
\|\phi\|_{Osc(D_{1/2})}=\sup_{x,y\in D_{1/2}}|\phi(x)-\phi(y)|\leq C(\Lambda,N)(\|\nabla \phi\|_{L^2(D)}+\|\tau(u)\|_{L^p(D)}).
\end{equation}
\end{thm}

\

\begin{prop}[Theorem 3.1 in \cite{chen2005regularity}]
There exists an $\epsilon_1>0$ depending on $N$ such that if $(\phi,\psi)$ is a smooth Dirac-harmonic map from the standard sphere $S^2$ to a compact Riemannian manifold $N$ satisfying
\[
\int_{S^2}(|d\phi|^2+|\psi|^4)dx<\epsilon_1,
\]
then $\phi$ is a constant map and $\psi\equiv 0$.
\end{prop}

\

\begin{thm}[Theorem 1.4 in \cite{jost-Liu-Zhu-02}]\label{thm:03}
Let $(\phi,\psi):\mathbb{R}^2_+\to N$ be a smooth Dirac-harmonic map with boundary data $\phi|_{\partial \mathbb{R}^2_+}=const.$ and
$\mathcal{B}\psi|_{\partial \mathbb{R}^2_+}=0$ and satisfying $$\int_{\mathbb{R}^2_+}|\nabla\phi|^2dx+\int_{\mathbb{R}^2_+}|\psi|^4dx <\infty.$$
Then $\phi$ is a constant map and $\psi\equiv 0$.
\end{thm}

\

\section{Energy identity}

\

In this section, we will prove our main Theorem \ref{thm:main-01}. Since the interior blow-up behavior was already studied in \cite{jost-Liu-Zhu-04}, we only need to consider the boundary blow-up behavior.

Firstly, we consider the following simpler case of a boundary blow-up point.

\begin{thm}\label{thm:02}
Let $\phi_n \in C^{2}(D^+_1(0),N)$, $\psi_n \in C^{1}(D^+_1(0),\Sigma D^+_1(0)\otimes \phi_n^{*}TN)$ be a sequence of approximate Dirac-harmonic maps satisfying
\begin{itemize}
\item[(a)]  $\ \|\phi_n\|_{W^{1,2}(D^+)}+\|\psi_n\|_{L^{4}(D^+)}+\|\tau_n\|_{L^{2}(D^+)}
    +\|h_n\|_{L^{4}(D^+)}\leq \Lambda,$ \\
\item[(b)] $\ (\phi_n,\psi_n)\rightharpoonup (\phi,\psi) \mbox{ weakly in }W_{loc}^{2,2}(D^+\setminus\{0\})\times W_{loc}^{1,2}(D^+\setminus\{0\})\ as\ n\to\infty$.
\end{itemize}
Then there exist a subsequence of $(\phi_n,\psi_n)$ (still denoted by $(\phi_n,\psi_n)$) and a nonnegative integer $L$ such that, for any $i=1,...,L$, there exist  points $x^i_n$,
positive numbers $\lambda^i_n$ and a nonconstant Dirac-harmonic sphere $(\sigma^i,\xi^i):S^2\to N$ such that:
\begin{itemize}
\item[(1)]  $\ x^i_n\to 0,\ \lambda^i_n\to 0$, as $n\to\infty$;\\
\item[(2)]  $\ \frac{dist(x^i_n,\partial^0D^+)}{\lambda^i_n}\to\infty\ $, as $n\to\infty$;\\
\item[(3)]  $\ \lim_{n\to\infty}\big(\frac{\lambda^i_n}{\lambda^j_n}+\frac{\lambda^j_n}{\lambda^i_n} +\frac{|x^i_n-x^j_n|}{\lambda^i_n+\lambda^j_n}\big)=\infty$ for any $i\neq j$;\\
\item[(4)] $\ (\sigma^i,\xi^i)$ is the weak limit of $(\phi_n(x^i_n+\lambda^i_nx),\psi_n(x^i_n+\lambda^i_nx))$ in $W^{1,2}_{loc}(\mathbb{R}^2)\times L^{4}_{loc}(\mathbb{R}^2)$;\\
\item[(5)]  \textbf{Energy identity:} we have
\begin{eqnarray}
\lim_{n\to\infty}E(\phi_n)&=&E(\phi)+\sum_{i=1}^{L}E(\sigma^i),\label{equation:thm02-1}\\
\lim_{n\to\infty}E(\psi_n)&=&E(\psi)+\sum_{i=1}^{L}E(\xi^i)\label{equation:thm02-2}.
\end{eqnarray}
\end{itemize}
\end{thm}
\begin{proof}
By assumption, without loss of generality, we may assume that $0$ is the only blow-up point of the sequence $\{(\phi_n,\psi_n)\}$ in $D^+$, $i.e.$
\begin{equation}\label{def:small-energy}
\liminf_{n\to\infty}E(\phi_n,\psi_n;D^+_r)\geq \frac{\overline{\epsilon}^2}{2}\mbox{ for all }r>0
\end{equation}
where $\overline{\epsilon}=\min\{\epsilon_0,\epsilon'_0\}$ and $\epsilon_0,\ \epsilon_0'$ are the constants in Theorem \ref{smthm-bound} and Theorem \ref{smthm}. By the standard argument of blow-up analysis (see e.g. \cite{DingWeiyueandTiangang, chen2005regularity}),  we can assume that, for any $n$, there exist sequences $x_n\to 0$ and $r_n\to 0$ such that
\begin{equation}\label{equation:25}
E(\phi_n,\psi_n;D^+_{r_n}(x_n))=\sup_{\substack{x\in D^+,r\leq r_n\\D^+_r(x)\subset D^+}}E(\phi_n,\psi_n;D^+_r(x))=\frac{\overline{\epsilon}^2}{4}.
\end{equation}

\

Firstly, we make a \textbf{Claim 1:} $\limsup_{n\to\infty}\frac{dist(x_n,\partial^0D^+)}{r_n}=\infty$.

If not, after taking a subsequence, we may assume $\lim_{n\to\infty}\frac{dist(x_n,\partial^0D^+)}{r_n}=a\geq0$.
Set
\begin{align*}
u_n(x,t):=\phi_n(x_n+r_nx),\
v_n(x,t):=\sqrt{r_n} \psi_n(x_n+r_nx),
\end{align*}
and
\[
B_n:=\{x\in\mathbb{R}^2|x_n+r_nx\in D^+\}.
\]
 Then
\[
B_n\to \mathbb{R}^2_a:=\{(x^1,x^2)|x^2\geq -a\},
\]
as $n\to\infty$. It is easy to see $(u_n,v_n)$ lives in $B_n$ and satisfies
\begin{eqnarray}\label{dh3}
\begin{cases}
r_n^2 \tau(\phi_n)=\Delta u_n+A(du_n,du_n)-  Re(P(\mathcal{A}(du_n
(e_{\alpha }),e_{\alpha }\cdot v_n );v_n)),\quad &in\quad B_n;\\
r_n^{\frac{4}{3}}h_n=\p v_n-\mathcal{A}(du_n (e_{\alpha }),e_{\alpha }\cdot v_n
),\quad &in\quad B_n,
\end{cases}
\end{eqnarray}
with the boundary data
\begin{eqnarray}
\begin{cases}
u_n(x)=\varphi(x_1+r_nx),\quad &if\quad x_n+r_nx\in \partial M;\\
\mathcal{B} v_n(x)=\sqrt{r_n}\mathcal{B} \chi(x_n+r_nx),\quad &if\quad x_n+r_nx\in \partial M.
\end{cases}
\end{eqnarray}
By \eqref{equation:25}, Theorem \ref{smthm} and Theorem \ref{smthm-bound}, we have
\begin{equation}\label{equation:04}
\|u_n\|_{W^{2,2}(D_{4R}(0)\cap B_n)}+\|v_n\|_{W^{1,2}(D_{4R}(0)\cap B_n)}\leq C(R,N)
\end{equation}
for any $D_R(0)\subset \mathbb{R}^2$ which implies
\begin{equation*}
\|u_n(x-(0,\frac{d_n}{r_n}))\|_{W^{2,2}(D_{3R}^+(0))} +\|v_n(x-(0,\frac{d_n}{r_n}))\|_{W^{1,2}(D_{3R}^+(0))}\leq C(R,N)
\end{equation*}
when $n,R$ are large, where $d_n:=dist(x_n,\partial^0D^+)$.

Then there exist a subsequence of $(u_n,v_n)$ (also denoted by $(u_n,v_n)$) and a Dirac-harmonic map $(\widetilde{u},\widetilde{v})\in W^{2,2}(\mathbb{R}_+^2)\times W^{1,2}(\mathbb{R}_+^2)$ with the boundary data $(\widetilde{u},\mathcal{B} \widetilde{v})|_{\partial \mathbb{R}^{2}_+}=(\varphi(x_0),0)$, such that for any $R>0$,
\begin{align*}
\lim_{n\to\infty}\|u_n(x-(0,\frac{d_n}{r_n}))-\widetilde{u}(x)\|_{W^{1,2}(D_{3R}^+(0))}&=0\\
\lim_{n\to\infty}\|v_n(x-(0,\frac{d_n}{r_n}))-\widetilde{v}(x)\|_{L^{4}(D_{3R}^+(0))}&=0.
\end{align*}

Set $\widetilde{u}^1(x):=\widetilde{u}(x+(0,a))$ and $\widetilde{v}^1(x):=\widetilde{v}(x+(0,a))$, then we get, for any $R>0$,
\begin{align*}
\lim_{n\to\infty}\|u_n(x)-\widetilde{u}^1(x)\|_{W^{1,2}(D_{2R}(0)\cap B_n\cap\mathbb{R}^2_a)}&=0\\
\lim_{n\to\infty}\|v_n(x)-\widetilde{v}^1(x)\|_{L^{4}(D_{2R}(0)\cap B_n\cap\mathbb{R}^2_a)}&=0.
\end{align*}
Combining this with \eqref{equation:04} and noting that the measure of $D_{2R}(0)\cap B_n\setminus\mathbb{R}^2_a$ goes to zero, we have
\begin{align*}
\lim_{n\to\infty}\|u_n(x)\|_{W^{1,2}(D_{R}(0)\cap B_n)}&=\|\widetilde{u}^1(x)\|_{W^{1,2}(D_{R}(0)\cap \mathbb{R}^2_a)}\\
\lim_{n\to\infty}\|v_n(x)\|_{L^{4}(D_{R}(0)\cap B_n)}&=\|\widetilde{v}^1(x)\|_{L^{4}(D_{R}(0)\cap \mathbb{R}^2_a)}.
\end{align*}

Therefore, by \eqref{equation:25}, we can obtain $E(\widetilde{u}^1,\widetilde{v}^1;D_1(0)\cap\mathbb{R}_a^2)=\frac{\overline{\epsilon}^2}{4}$. However, by Theorem \ref{thm:03}, we know $\widetilde{u}^1$ is a constant map and $\widetilde{v}^1\equiv 0$. This is a contradiction. We proved \textbf{Claim 1}.

\

Under the assumption $\limsup_{n\to\infty}\frac{dist(x_n,\partial^0D^+)}{r_n}=\infty$, we can see that $(u_n,v_n)$ lives in $B_n$ which tends to $\mathbb{R}^2$ as $n\to\infty$.
Moreover, for any $x\in\mathbb{R}^2$, when $n$ is sufficiently large, by \eqref{equation:25}, we have
\begin{equation}
E(u_n,v_n;D_1(x))\leq \frac{\overline{\epsilon}^2}{4}.
\end{equation}
According to Theorem \ref{smthm}, there exist a subsequence of $(u_n,v_n)$ (we still denote it by $(u_n,v_n)$) and a Dirac-harmonic map $(u^1(x),v^1(x))\in W^{2,2}(\mathbb{R}^2,N)\times W^{1,2}(\mathbb{R}^2,\Sigma \mathbb{R}^2 \otimes (u^1)^{*}TN)$ such that
\begin{align}\label{equation:02}
u_n(x)\to u^1(x) \mbox{ in }W^{1,2}_{loc}(\mathbb{R}^2),\quad
v_n(x)\to v^1(x) \mbox{ in }L^{4}_{loc}(\mathbb{R}^2),
\end{align}
as $n\to\infty$. Besides, we know $E(u^1,v^1;D_1(0))=\frac{\overline{\epsilon}^2}{4}$. By the standard theory of Dirac-harmonic maps \cite{chen2005regularity}, $(u^1(x),v^1(x))$ can be extended to a nontrivial Dirac-harmonic sphere which is usually called the first bubble.

\

By the standard induction argument in \cite{DingWeiyueandTiangang}, we only need to prove the theorem in the case where there is only one bubble. For the more bubbles case, i.e. the bubble tree, we just need to distinguish ``neck domains" which is almost the same as in the blow-up theory of approximate harmonic maps. See \cite{Li-Wang2,Chen-Li} for details. Then we can estimate the energy concentration on each ``neck domain" by using the proof of the  one bubble case.

\

Under this assumption, we have the following:

\

\noindent\textbf{Claim 2:} for any $\epsilon>0$, there exist $\delta>0$ and $R>0$ such that
\begin{equation}\label{equation:assumption-small-bound}
E(\phi_n,\psi_n;D^+_{8t}(x_n)\setminus D^+_{t}(x_n))\leq \epsilon^2 \mbox{ for any } t\in(\frac{1}{2}r_nR,2\delta)
\end{equation}
when $n$ is large enough.

\

In fact, if \eqref{equation:assumption-small-bound} is not true, then we can find $t_n\to 0$, such that $\lim_{n\to \infty}\frac{t_n}{r_n}=\infty$ and $\epsilon'>0$ such that
\begin{equation}\label{equation:10}
E(\phi_n,\psi_n;D^+_{8t_n}(x_n)\setminus D^+_{t_n}(x_n))\geq \epsilon'>0.
\end{equation}

Passing to a subsequence, we may assume $\lim_{n\to\infty}\frac{d_n}{t_n}=b\in [0,\infty].$ For simplicity of notation, we also denote $$u_n(x):=\phi_n(x_n+t_nx),v_n(x):=\sqrt{t_n}\psi_n(x_n+t_nx).$$ Denoting $B_n':=\{x\in\mathbb{R}^2|x_n+t_nx\in D^+\}$, then it is easy to see that $(u_n(x),v_n(x))$ lives in $B_n'$ and $0$ is also an energy concentration point for $(u_n,v_n)$. We have to consider the following two cases:

\

$\textbf{(a)}$ $b<\infty$.

\

Then $B_n'$ tends to $\mathbb{R}^2_b$ as $n\to \infty$. Here, we also need to consider two cases.

\

$\textbf{(a-1)}$ $(u_n,v_n)$ has no other energy concentration points except $0$.

\

By Theorem \ref{smthm}, Theorem \ref{smthm-bound} and the proof of \textbf{Claim 1}, there exists a Dirac-harmonic map $(u,v):\mathbb{R}^2_b\to N$ with boundary data $u|_{\partial \mathbb{R}^2_b}=\varphi(0)$, $\mathcal{B}v|_{\partial \mathbb{R}^2_b}=0$ satisfying, passing to a subsequence, for any $\lambda, R>0$, there hold
\begin{align*}
\lim_{n\to\infty}\|u_n(x)-u(x)\|_{W^{1,2}(D_{2R}(0)\cap B_n'\cap\mathbb{R}^2_b\setminus{D_\lambda(0)})}&=0\\
\lim_{n\to\infty}\|v_n(x)-v(x)\|_{L^{4}(D_{2R}(0)\cap B_n'\cap\mathbb{R}^2_b\setminus{D_\lambda(0)})}&=0,
\end{align*}
and
\begin{align*}
\lim_{n\to\infty}\|u_n(x)\|_{W^{1,2}(D_{R}(0)\cap B_n'\setminus{D_\lambda(0)})}&=\|u(x)\|_{W^{1,2}(D_{R}(0)\cap \mathbb{R}^2_b\setminus{D_\lambda(0)})}\\
\lim_{n\to\infty}\|v_n(x)\|_{L^{4}(D_{R}(0)\cap B_n'\setminus{D_\lambda(0)})}&=\|v(x)\|_{L^{4}(D_{R}(0)\cap \mathbb{R}^2_b\setminus{D_\lambda(0)})}.
\end{align*}
According to \eqref{equation:10}, we have
\begin{align*}
E(u,v;(D_8(0)\setminus D_1(0))\cap\mathbb{R}^2_b)=\lim_{n\to\infty}E(u_n,v_n;(D_8(0)\setminus D_1(0))\cap B_n')\geq \epsilon'.
\end{align*}
However, Theorem \ref{thm:03} tells us that $u$ is a constant map and $v\equiv 0$. This is a contradiction.

\

\noindent$\textbf{(a-2)}$ $(u_n,v_n)$ has another energy concentration point $p\neq 0$.

\

Without loss of generality, we may assume $p$ is the only energy concentration point in $D_{r_0}(p)$ for some $r_0>0$.
By the standard argument of blow-up analysis, there exist sequences $x_n'\to p$ and $r_n'\to 0$ such that
\begin{equation}\label{equation:13}
E(u_n,v_n;D_{r'_n}(x'_n)\cap B_n')=\sup_{\substack{x\in D_{r_0}(p),r\leq r_n'\\D_r(x)\subset D_{r_0}(p)}}E(u_n,v_n;D_r(x)\cap B_n')=\frac{\overline{\epsilon}^2}{4}.
\end{equation}
By \eqref{equation:25}, we have $r_n't_n\geq r_n$ and taking a subsequence, we may assume
$$\lim_{n\to\infty}\frac{d_n}{r'_nt_n}=d\in [0,\infty].$$
Furthermore, we know $d$ must be $\infty$ (the proof is the same as for  \textbf{Claim 1}). Then similar to the process of constructing the first bubble, there exists a nontrivial Dirac-harmonic map $(u^2(x),v^2(x))\in W^{2,2}(\mathbb{R}^2,N)\times W^{1,2}(\mathbb{R}^2,\Sigma \mathbb{R}^2 \otimes (u^2)^{*}TN)$ such that
\begin{align*}
u_n(x_n'+r_n'x)\to u^2(x) \mbox{ in }W^{1,2}_{loc}(\mathbb{R}^2),\quad
\sqrt{r_n'}v_n(x_n'+r_n'x)\to v^2(x) \mbox{ in }L^{4}_{loc}(\mathbb{R}^2),
\end{align*}
as $n\to\infty$. This is $$\phi_n(x_n+t_nx_n'+t_nr_n'x)\to v^2(x)\ in\ W^{1,2}_{loc}(\mathbb{R}^2)\ and\ \sqrt{t_nr_n'}\psi_n(x_n+t_nx_n'+t_nr_n'x)\to v^2(x)\ in\ L^{4}_{loc}(\mathbb{R}^2).$$ Thus, $(u^2,v^2)$ is also a bubble for the sequence $(\phi_n,\psi_n)$. This is a contradiction to the one bubble assumption.

\

\noindent$\textbf{(b)}$ $b=\infty$.

\

In this case, $B'_n$ will tend to $\mathbb{R}^2$ as $n\to \infty$. Again, we need to consider the following two cases.

\

\noindent$\textbf{(b-1)}$ $(u_n,v_n)$ has no other energy concentration points except $0$.

\

According to \eqref{equation:10}, Theorem \ref{smthm}, Theorem \ref{smthm-bound} and the process of constructing the first bubble, we know that there exists a nontrivial Dirac-harmonic map $(u^2,v^2):\mathbb{R}^2\to N$ such that, passing to a subsequence, $$u_n(x)\to u^2(x)\ in\ W^{1,2}_{loc}(\mathbb{R}^2\setminus \{0\})\ and\ v_n(x)\to v^2(x)\ in\ L^{4}_{loc}(\mathbb{R}^2\setminus \{0\}),$$ as $n\to\infty$. Then, we get the second bubble $(u^2(x),v^2(x))$ which  contradicts the ``one bubble" assumption.

\

\noindent$\textbf{(b-2)}$ $(u_n,v_n)$ has another energy concentration point $p\neq 0$.

\

Similar to  $\textbf{Case (a-2)}$, there exist sequences $x_n'\to p$ and $r_n'\to 0$ satisfying \eqref{equation:13} and $$\lim_{n\to\infty}\frac{d_n}{r'_nt_n}=\infty.$$ Moreover, by the process of constructing the first bubble, there exists a nontrivial Dirac-harmonic map $(u^2,v^2):\mathbb{R}^2\to N$  such that, as $n\to\infty$, $$u_n(x_n'+r_n'x)\to v^2(x)\ \text{in }\ W^{1,2}_{loc}(\mathbb{R}^2)\ \text{and} \ \sqrt{r_n'}v_n(x_n'+r_n'x)\to v^2(x)\ \text{in}\ L^{4}_{loc}(\mathbb{R}^2)$$ that is $$\phi_n(x_n+t_nx_n'+t_nr_n'x)\to v^2(x)\ \text{in}\ W^{1,2}_{loc}(\mathbb{R}^2)\ \text{and}\ \sqrt{t_nr_n'}\psi_n(x_n+t_nx_n'+t_nr_n'x)\to v^2(x)\ \text{in}\ L^{4}_{loc}(\mathbb{R}^2).$$ So, we get the second bubble $(u^2(x),v^2(x))$. This  also contradicts the  ``one bubble " assumption. Thus, we proved \textbf{Claim 2}.

\

Under the ``one bubble" assumption, by \eqref{equation:02}, it is easy to see that energy identity \eqref{equation:thm02-1} and \eqref{equation:thm02-2} are equivalent to
\begin{equation}\label{equation:energy-identity-01}
\lim_{R\to\infty}\lim_{\delta\to 0}\lim_{n\to \infty}E(\phi_n;D^+_\delta(x_n)\setminus D^+_{r_nR}(x_n))=0
\end{equation}
and
\begin{equation}\label{equation:energy-identity-02}
\lim_{R\to\infty}\lim_{\delta\to 0}\lim_{n\to \infty}E(\psi_n;D^+_\delta(x_n)\setminus D^+_{r_nR}(x_n))=0.
\end{equation}

Without loss of generality, we may assume $\delta=2^{m_n}r_nR$ for some positive integer $m_n$ which tends to $\infty$ as $n\to\infty$. We denote $P_i:=D^+_{2^ir_nR}(x_n)\setminus D^+_{2^{i-1}r_nR}(x_n)$.

Firstly we use a finite decomposition argument that is similar to those in \cite{Ye,zhao2007energy} to separate $\Sigma:=D^+_{\delta}(x_n)\setminus D^+_{r_nR}(x_n)$ into finite parts
\begin{align*}
\Sigma=\cup_{j=1}^{s_n}Q_j,\quad Q_j:=\cup_{i=k_{j-1}+1}^{k_j}P_i, \quad 0=k_0<k_1<,...,<k_{s_n}=m_n
\end{align*}
such that $s_n\leq S$ and
\begin{align}\label{inequality:11}
E(\phi_n,\psi_n;Q_j)\leq \frac{1}{C_1(N)},\quad j=1,...,s_n,
\end{align}
where $C_1(N)>0$ is a constant depending only on $N$ to be determined later and $S$ is a uniform integer for all $n$ large enough.

From \eqref{equation:assumption-small-bound}, for any $\epsilon<\frac{1}{2C_1(N)}$, we have
\begin{equation*}
E(\phi_n,\psi_n;P_i)<\epsilon<\frac{1}{2C_1(N)},\quad i=1,...,m_n
\end{equation*}
when $n$ is large.

If
\[
E(\phi_n,\psi_n;\Sigma)\leq \frac{1}{C_1(N)},
\]
let $k_1=m_n$ and then $Q_1=\Sigma$. Otherwise, we can choose an integer $1\leq k_1<m_n$ such that
\begin{align*}
\frac{1}{2C_1(N)}<E(\phi_n,\psi_n;Q_1)\leq \frac{1}{C_1(N)} \quad \text{and} \quad E(\phi_n,\psi_n;Q_1\cup P_{k_1+1})> \frac{1}{C_1(N)}.
\end{align*}
This is the first step of the division. Inductively, suppose that $k_j$ is chosen such that $E(\phi_n,\psi_n;Q_j)\leq \frac{1}{C_1(N)}$. If
\[
E(\phi_n,\psi_n;\cup_{i=k_j+1}^{m_n}P_i)\leq \frac{1}{C_1(N)},
\]
let $k_{j+1}=m_n$, thus $Q_{j+1}=\cup_{i=k_j+1}^{m_n}P_i$. If not, then similar to the first step, we can find $k_j<k_{j+1}<m_n$ such that
\begin{align*}
\frac{1}{2C_1(N)}<E(\phi_n,\psi_n;Q_{j+1})\leq \frac{1}{C_1(N)} \quad and \quad E(\phi_n,\psi_n;Q_{j+1}\cup P_{k_{j+1}+1})> \frac{1}{C_1(N)}.
\end{align*}
Since $E(\phi_n,\psi_n)$ is uniformly bounded by $\Lambda$, we will finish our division after at most $S=[2C_1(N)\Lambda]+1$ steps. So we have finished the division.

Take a cut-off function $\eta\in C_0^\infty(D^+_{2^{k_j+1}r_nR}(x_n)\setminus D^+_{2^{k_{j-1}-1}r_nR}(x_n))$ such that $0\leq\eta \leq1$ and $\eta|_{D^+_{2^{k_j}r_nR}(x_n)\setminus D^+_{2^{k_{j-1}}r_nR}(x_n)}\equiv 1$ and
\begin{align*}
|\nabla\eta|\leq\frac{C}{2^{k_j}r_nR} \quad &on\quad D^+_{2^{k_j+1}r_nR}(x_n)\setminus D^+_{2^{k_{j}}r_nR}(x_n)\quad and\\
|\nabla\eta|\leq\frac{C}{2^{k_{j-1}}r_nR} \quad &on\quad D^+_{2^{k_{j-1}}r_nR}(x_n)\setminus D^+_{2^{k_{j-1}-1}r_nR}(x_n).
\end{align*}

By the standard elliptic estimates, we have
\begin{align*}
&\|\eta\psi_n\|_{W^{1,4/3}(D_1^+)}\\&\leq
C\|\eta\p\psi_n+\nabla\eta\cdot\psi_n\|_{L^{\frac{4}{3}}(D_1^+)} +C\|\eta\mathcal{B}\chi\|_{W^{1/4,4/3}(\partial D^+_1)}\\
&\leq
\frac{1}{4}C(N)(\||d\phi_n||\eta\psi_n|\|_{L^{\frac{4}{3}}(\Sigma)}+\|\eta| h_n|\|_{L^{\frac{4}{3}}(\Sigma)})
+C\||\nabla\eta||\psi_n|\|_{L^{\frac{4}{3}}(\Sigma)}+C\|\eta\mathcal{B}\chi\|_{W^{1/4,4/3}(\partial^0 D^+_\delta)}\\
&\leq
\frac{1}{4}C(N)\|d\phi_n\|_{L^2(D_{2^{k_j+1}r_nR}(x_n)\setminus D_{2^{k_{j-1}-1}r_nR}(x_n))}\|\eta\psi_n\|_{L^4(\Sigma)}+C\| h_n\|_{L^{\frac{4}{3}}(\Sigma)}\\
&\quad+C\|\nabla\eta\psi_n\|_{L^{\frac{4}{3}}(P_{k_{j-1}}\cup P_{k_j+1})}+C\|\eta\mathcal{B}\chi\|_{W^{1/4,4/3}(\partial^0 D^+_\delta)}\\
&\leq
\frac{1}{4}C(N)
\frac{2}{\sqrt{C_1(N)}}\|\eta\psi_n\|_{L^4(\Sigma)}+C\|\psi_n\|_{L^{4}(P_{k_{j-1}}\cup P_{k_j+1})}+C\| h_n\|_{L^{\frac{4}{3}}(\Sigma)}+C\|\eta\mathcal{B}\chi\|_{W^{1/4,4/3}(\partial^0 D^+_\delta)},
\end{align*}
where the last inequality is from \eqref{equation:assumption-small-bound} and \eqref{inequality:11}. Then, taking $C_1(N)=C^2(N)+1$, by \eqref{equation:assumption-small-bound} and Sobolev embedding, we have
\begin{align*}
\|\psi_n\|_{L^4(Q_j)}+\|\nabla\psi_n\|_{L^{4/3}(Q_j)}
&\leq
C\|\psi_n\|_{L^{4}(P_{k_{j-1}}\cup P_{k_j+1})}+C\| h_n\|_{L^{\frac{4}{3}}(\Sigma)}+C\|\eta\mathcal{B}\chi\|_{W^{1/4,4/3}(\partial^0 D^+_\delta)}\\&\leq C(\Lambda,\|\chi\|_{C^1})(\sqrt{\epsilon}+\delta).
\end{align*}
So,
\begin{align*}
\|\psi_n\|_{L^4(\Sigma)}+\|\nabla\psi_n\|_{L^{4/3}(\Sigma)}
&\leq
\sum_{j=1}^{s_n}(\|\psi_n\|_{L^4(Q_j)}+\|\nabla\psi_n\|_{L^{4/3}(Q_j)})\leq
CS(\sqrt{\epsilon}+\delta).
\end{align*}
This is \eqref{equation:energy-identity-02}.

\

Suppose $x_n'\in\partial^0 D^+$ is the projection of $x_n$, $i.e.$ $d_n=dist(x_n,\partial^0D^+)=|x_n-x_n'|$. Similar to the boundary blow-up cases for approximate harmonic maps studied in \cite{jost-Liu-Zhu-03,jost-Liu-Zhu-05}, we decompose the neck domain $D^+_\delta(x_n)\setminus D^+_{r_nR}(x_n)$ as follows
\begin{align*}
D^+_\delta(x_n)\setminus D^+_{r_nR}(x_n)&=D^+_\delta(x_n)\setminus D^+_{\frac{\delta}{2}}(x'_n)\cup D^+_{\frac{\delta}{2}}(x'_n)\setminus D^+_{2d_n}(x'_n)\\
&\quad\cup D^+_{2d_n}(x'_n)\setminus D^+_{d_n}(x_n)\cup D^+_{d_n}(x_n)\setminus D^+_{r_nR}(x_n)\\
&:=\Omega_1\cup\Omega_2\cup\Omega_3\cup\Omega_4,
\end{align*}
when $n$ is large.

Since $\lim_{n\to\infty}d_n=0$ and $\lim_{n\to\infty}\frac{d_n}{r_n}=\infty$, when $n$ is large enough, it is easy to see that
\[
\Omega_1\subset D^+_\delta(x_n)\setminus D^+_{\frac{\delta}{4}}(x_n),\quad and \quad\Omega_3\subset D^+_{4d_n}(x_n)\setminus D^+_{d_n}(x_n).
\]

Moreover, for any $d_n\leq t\leq \delta$, there holds
\[
D^+_{2t}(x'_n)\setminus D^+_{t}(x'_n)\subset D^+_{4t}(x_n)\setminus D^+_{t/2}(x_n).
\]

By assumption \eqref{equation:assumption-small-bound}, we have
\begin{equation}\label{inequality:12}
E(\phi_n,\psi_n;\Omega_1)+E(\phi_n,\psi_n;\Omega_3)\leq \epsilon^2
\end{equation}
and
\begin{equation}
E(\phi_n,\psi_n;D^+_{2t}(x'_n)\setminus D^+_{t}(x'_n))\leq \epsilon^2 \mbox{ for any } t\in(d_n, \delta).
\end{equation}

By \eqref{equation:assumption-small-bound}, Theorem \ref{smthm}, Theorem \ref{smthm-bound} and the standard scaling argument, we get
\begin{align}\label{inequality:06}
&Osc_{D^+_{2t}(x'_n)\setminus D^+_{t}(x'_n)}\phi_n\notag\\&\leq C(\|\nabla \phi_n\|_{L^2(D^+_{4t}(x'_n)\setminus D^+_{t/2}(x'_n))}+\| \psi_n\|_{L^4(D^+_{4t}(x'_n)\setminus D^+_{t/2}(x'_n))}+\|\nabla \varphi\|_{L^2(D^+_{4t}(x'_n)\setminus D^+_{t/2}(x'_n))}\notag\\
&\quad+t\|\nabla^2 \varphi\|_{L^2(D^+_{4t}(x'_n)\setminus D^+_{t/2}(x'_n))}+t\|\tau(u_n)\|_{L^2(D^+_{4t}(x'_n)\setminus D^+_{t/2}(x'_n))})\notag\\
&\leq C(\sqrt{\epsilon}+\delta),
\end{align}
for any $t\in(2r_nR, \frac{1}{2}\delta)$, where $C=C(\Lambda,N,\|\varphi\|_{C^2},\|\chi\|_{C^1})$ is a positive constant.

Noting that $\Omega_4=D^+_{d_n}(x_n)\setminus D^+_{r_nR}(x_n)=D_{d_n}(x_n)\setminus D_{r_nR}(x_n)$, by the energy identity of approximate Dirac-harmonic maps with interior blow-up points (see Theorem 1.2 in \cite{jost-Liu-Zhu-04}), there holds
\begin{equation}\label{equation:26}
\lim_{R\to\infty}\lim_{n\to 0}E(u_n;D_{d_n}(x_n)\setminus D_{r_nR}(x_n))=0.
\end{equation}

Therefore, we just need to estimate the energy concentration in $\Omega_2$. Here, we use a similar method as in \cite{jost-Liu-Zhu-03,jost-Liu-Zhu-05}.

Define $\widehat{\Omega_2}:=D_{\frac{\delta}{2}}(x'_n)\setminus D_{2d_n}(x'_n)$, $\Phi_n(x):=\phi_n(x)-\varphi(x),\ x\in \Omega_2$ and
\begin{align}\label{equation:15}
\widehat{\Phi_n}(x):=
\begin{cases}
\Phi_n(x),\ &x\in \Omega_2,\\
-\Phi_n(x'),\ &x\in \widehat{\Omega_2}\setminus\Omega_2,
\end{cases}
\end{align}
where $x=(x^1,x^2)$ and $x'=(x^1,-x^2)$. It is easy to see that $\widehat{\Phi_n}(x)\in W^{2,\infty}(\widehat{\Omega_2})$ and satisfies the following equation
\begin{align}\label{equation:16}
\Delta\widehat{\Phi_n}(x)=
\begin{cases}
\Delta\phi_n(x)-\Delta\varphi(x),\ &x\in \Omega_2,\\
-\Delta\phi_n(x')+\Delta\varphi(x'),\ &x\in \widehat{\Omega_2}\setminus\Omega_2,
\end{cases}
\end{align}
where $\Delta\phi_n(x)=-A(d \phi_n,d \phi_n)(x)+Re\left( P(\mathcal{A}(d\phi_n(e_\alpha),e_\alpha\cdot\psi_n);\psi_n)\right)(x)+\tau(u_n)(x)$.

Without loss of generality, we may also assume $\frac{1}{2}\delta=2^{m_n'}(2d_n)$, where $m_n'$ is a positive integer which tends to $\infty$ as $n\to\infty$. Setting $P_i':=D^+_{2^{i+1}d_n}(x_n')\setminus D^+_{2^{i}d_n}(x_n')$ and $\widehat{P_i'}:=D_{2^{i+1}d_n}(x_n')\setminus D_{2^{i}d_n}(x_n')$,

Set
\[
\widehat{\Phi_n}^*(r):=\frac{1}{2\pi }\int_{0}^{2\pi}\widehat{\Phi_n}(r,\theta)d\theta,
\]
where $(r,\theta)$ are the polar coordinates at $x_n'$. By \eqref{inequality:06} and \eqref{equation:15}, we have
\begin{align}\label{equation:18}
\|\widehat{\Phi_n}(x)-\widehat{\Phi_n}^*(x)\|_{L^\infty(\widehat{\Omega_2})}&\leq \sup_{1\leq i\leq m_n'}\|\widehat{\Phi_n}(x)-\widehat{\Phi_n}^*(x)\|_{L^\infty(\widehat{P_i'})} \leq \sup_{1\leq i\leq m_n'}\|\widehat{\Phi_n}(x)\|_{Osc(\widehat{P_i'})}\notag\\
&\leq 2\sup_{1\leq i\leq m_n'}\|\Phi_n(x)\|_{Osc(P_i')}\leq 2\sup_{1\leq i\leq m_n'}\|\phi_n(x)\|_{Osc(P_i')}+C\delta\|\nabla\varphi\|_{L^\infty}\notag\\& \leq C(N,\Lambda,\|\varphi\|_{C^{2}},\|\chi\|_{C^1})( \sqrt{\epsilon}+\delta).\end{align}

Integrating by parts, we get
\begin{equation*}
\int_{\widehat{P_i'}}\nabla\widehat{\Phi_n}\nabla(\widehat{\Phi_n}-\widehat{\Phi_n}^*)dx=
\int_{\partial \widehat{P_i'}}(\widehat{\Phi_n}-\widehat{\Phi_n}^*)\frac{\partial \widehat{\Phi_n}}{\partial r}-\int_{\widehat{P_i'}}(\widehat{\Phi_n}-\widehat{\Phi_n}^*)\Delta\widehat{\Phi_n}dx.
\end{equation*}
On the one hand, we have
\begin{align*}
\int_{\widehat{P_i'}}\nabla\widehat{\Phi_n}\nabla(\widehat{\Phi_n}-\widehat{\Phi_n}^*)dx&=
\int_{\widehat{P_i'}}|\nabla\widehat{\Phi_n}|^2dx-\int_{\widehat{P_i'}}\frac{\partial \widehat{\Phi_n}}{\partial r}\frac{\partial \widehat{\Phi_n}^*}{\partial r}dx\\
&\geq
\int_{\widehat{P_i'}}|\nabla\widehat{\Phi_n}|^2dx-(\int_{\widehat{P_i'}}|\frac{\partial \widehat{\Phi_n}}{\partial r}|^2dx)^{\frac{1}{2}}(\int_{\widehat{P_i'}}|\frac{\partial \widehat{\Phi_n}^*}{\partial r}|^2dx)^{\frac{1}{2}}\\
&\geq
\int_{\widehat{P_i'}}|\nabla\widehat{\Phi_n}|^2dx-\int_{\widehat{P_i}}|\frac{\partial \widehat{\Phi_n}}{\partial r}|^2dx\\
&=
\frac{1}{2}\int_{\widehat{P_i'}}|\nabla\widehat{\Phi_n}|^2dx -\int_{\widehat{P_i'}}(|\frac{\partial \widehat{\Phi_n}}{\partial r}|^2-\frac{1}{2}|\nabla\widehat{\Phi_n}|^2 )dx\\
&=
\int_{P_i'}|\nabla\Phi_n|^2dx-2\int_{P_i'}(|\frac{\partial \Phi_n}{\partial r}|^2-\frac{1}{2}|\nabla\Phi_n|^2)dx.
\end{align*}
By direct computation, we obtain
\begin{align*}
&\int_{P_i'}|\nabla\Phi_n|^2dx-2\int_{P_i'}(|\frac{\partial \Phi_n}{\partial r}|^2-\frac{1}{2}|\nabla\Phi_n|^2)dx\\&=\int_{P_i'}|\nabla \phi_n|^2dx-2\int_{P_i'}(|\frac{\partial \phi_n}{\partial r}|^2-\frac{1}{2}|\nabla \phi_n|^2)dx+4\int_{P_i'}(\frac{\partial \phi_n}{\partial r}\frac{\partial \varphi}{\partial r}-\nabla \phi_n\nabla\varphi)dx\\&\quad+2\int_{P_i'}(|\nabla \varphi|^2-|\frac{\partial \varphi}{\partial r}|^2)dx\\
&\geq
\int_{P_i'}|\nabla \phi_n|^2dx-2\int_{P_i'}(|\frac{\partial \phi_n}{\partial r}|^2-\frac{1}{2}|\nabla \phi_n|^2)dx-C2^id_n.
\end{align*}

On the other hand, by \eqref{equation:16} and \eqref{equation:18}, we have
\begin{align*}
\int_{\widehat{P_i'}}(\widehat{\Phi_n}-\widehat{\Phi_n}^*)\Delta\widehat{\Phi_n}dx&\leq
C(\sqrt{\epsilon}+\delta)\int_{P_i'}|d\phi_n|^2 dx +C(\sqrt{\epsilon}+\delta)\int_{P_i'}|d\phi||\psi_n|^2dx\\&\quad+C(\sqrt{\epsilon}+\delta)\int_{P_i'} (|\tau_n|+|\Delta\varphi|)dx\\
&\leq
C(\sqrt{\epsilon}+\delta)\int_{P_i'}|d\phi_n|^2dx+C(\sqrt{\epsilon}+\delta)\int_{P_i'}|\psi_n|^4 dx +C(\sqrt{\epsilon}+\delta)2^id_n.
\end{align*}

From the above, by Corollary \ref{cor:poho-bound} (taking $\varepsilon=\frac{1}{2}$), we get
\begin{align*}
\int_{P_i'}|d\phi_n|^2dx\leq&
\int_{\partial \widehat{P_i'}}(\widehat{\Phi_n}-\widehat{\Phi_n}^*)\frac{\partial \widehat{\Phi_n}}{\partial r}+
\int_{P_i'}(|\frac{\partial \phi_n}{\partial r}|^2-\frac{1}{2}|\nabla \phi_n|^2)dx+C(\sqrt{\epsilon}+\delta)\int_{P_i'}|d\phi_n|^2dx\\&+C(\sqrt{\epsilon}+\delta)\int_{P_i'}|\psi_n|^4dx
+C2^id_n\notag\\
\leq&
\int_{\partial \widehat{P_i'}}(\widehat{\Phi_n}-\widehat{\Phi_n}^*)\frac{\partial \widehat{\Phi_n}}{\partial r}+(\frac{1}{2}+C(\sqrt{\epsilon}+\delta))\int_{P_i'}|d\phi_n|^2dx+C\int_{P_i'}|\psi_n|^4dx
\\&+C\int_{P_i'}|\nabla\psi_n|^{\frac{4}{3}}dx+
C2^id_n.
\end{align*}

Summing $i$ from $1$ to $m_n'$, we get
\begin{align}\label{inequality:14}
(\frac{1}{2}-C(\sqrt{\epsilon}+\delta))\int_{\Omega_2}|\nabla\phi_n|^2dx
&\leq
\int_{\partial D_{\delta/2}(x_n')}(\widehat{\Phi_n}-\widehat{\Phi_n}^*)\frac{\partial \widehat{\Phi_n}}{\partial r}-\int_{\partial D_{2d_n}(x_n')}(\widehat{\Phi_n}-\widehat{\Phi_n}^*)\frac{\partial \widehat{\Phi_n}}{\partial r}\notag\\&\quad+C\int_{\Omega_2}|\nabla\psi_n|^{4/3}dx+C\int_{\Omega_2}|\psi_n|^4dx+
C\delta.
\end{align}
As for the boundary term, by trace theory, we have
\begin{align*}
\int_{\partial D_{\delta/2}(x_n')}(\widehat{\Phi_n}-\widehat{\Phi_n}^*)\frac{\partial \widehat{\Phi_n}}{\partial r}&\leq C(\sqrt{\epsilon}+\delta)\int_{\partial D_{\delta/2}(x_n')}|\nabla \widehat{\Phi_n}|\\
&\leq C(\sqrt{\epsilon}+\delta)\int_{\partial^+ D_{\delta/2}(x_n')}(|\nabla \phi_n|+|\nabla\varphi|)\\
&\leq
C(\sqrt{\epsilon}+\delta)\left(\|\nabla \phi_n\|_{L^2(D^+_{\delta}\setminus D^+_{\frac{1}{4}\delta} )}+\delta\|\nabla^2 \phi_n\|_{L^2(D^+_{\delta}\setminus D^+_{\frac{1}{4}\delta} )}+1\right)\\
&\leq
C(\sqrt{\epsilon}+\delta)\big(\|\nabla \phi_n\|_{L^2(D^+_{\frac{4}{3}\delta}\setminus D^+_{\frac{1}{6}\delta} )}+\| \psi_n\|_{L^4(D^+_{\frac{4}{3}\delta}\setminus D^+_{\frac{1}{6}\delta} )}\\&\quad+\|\nabla \varphi\|_{L^2(D^+_{\frac{4}{3}\delta}\setminus D^+_{\frac{1}{6}\delta} )}+\delta\|\nabla^2\varphi\|_{L^2(D^+_{\frac{4}{3}\delta}\setminus D^+_{\frac{1}{6}\delta} )}\\&\quad+\delta\|\tau_n\|_{L^2(D^+_{\frac{4}{3}\delta}\setminus D^+_{\frac{1}{6}\delta} )}+1\big)\\&\leq C(\sqrt{\epsilon}+\delta),
\end{align*}
where the last second inequality can be derived from Theorem \ref{smthm-bound}.

Also, there holds
\begin{align*}
\int_{\partial D_{2d_n}}(\widehat{\Phi_n}-\widehat{\Phi_n}^*)\frac{\partial \widehat{\Phi_n}}{\partial r}
\leq C(\sqrt{\epsilon}+\delta).
\end{align*}
Putting these in \eqref{inequality:14} and taking $\epsilon$ and $\delta$ sufficient small, we have
\begin{align}
\int_{\Omega_2}|\nabla\phi_n|^2dx
&\leq C\int_{\Omega_2}|\nabla\psi_n|^{4/3}dx+C\int_{\Omega_2}|\psi_n|^4dx+
C(\sqrt{\epsilon}+\delta).
\end{align}
Combining this with \eqref{inequality:12}, \eqref{equation:26} and \eqref{equation:energy-identity-02}, we will obtain \eqref{equation:energy-identity-01} and we finished the proof of Theorem \ref{thm:02}.
\end{proof}

\

\begin{proof}[\textbf{Proof of Theorem \ref{thm:main-01}:}]
It is easy to see that Theorem \ref{thm:main-01} is a consequence of the interior blow-up case, i.e. Theorem 1.2 in \cite{jost-Liu-Zhu-04} and the model case of boundary blow-ups, i.e. Theorem \ref{thm:02}.
\end{proof}

\vskip1cm


\providecommand{\bysame}{\leavevmode\hbox to3em{\hrulefill}\thinspace}
\providecommand{\MR}{\relax\ifhmode\unskip\space\fi MR }
\providecommand{\MRhref}[2]{%
  \href{http://www.ams.org/mathscinet-getitem?mr=#1}{#2}
}
\providecommand{\href}[2]{#2}

\end{document}